\documentclass[oneside]{amsart}
\usepackage[pdfdisplaydoctitle=true,
            colorlinks=true,
            urlcolor=blue,
            citecolor=blue,
            linkcolor=blue,
            pdfstartview=FitH,
            pdfpagemode= UseNone,
            bookmarksnumbered=true]{hyperref}
\usepackage[symbol]{footmisc}
\usepackage[normalem]{ulem}
\newtheorem{theorem}{Theorem}[section]
\newtheorem{lemma}[theorem]{Lemma}
\newtheorem{proposition}[theorem]{Proposition}

\theoremstyle{definition}

\newtheorem*{theorem*}{Theorem}

\theoremstyle{remark}
\newtheorem{remark}[theorem]{Remark}

\numberwithin{equation}{section}

\usepackage{amsmath}
\usepackage{amsfonts}
\usepackage{amssymb}
\usepackage[dvipsnames]{xcolor}
\usepackage{mathrsfs}
\usepackage{enumerate}
\usepackage{mathtools}
\mathtoolsset{showonlyrefs}

\newcommand{\abs}[1]{\left\lvert#1\right\rvert}
\newcommand{\sabs}[1]{\lvert#1\rvert}

\newcommand{\norm}[1]{\left\lVert#1\right\rVert}
\newcommand{\snorm}[1]{\lVert#1\rVert}

\newcommand{\ip}[1]{\left\langle #1 \right\rangle}

\newcommand{\R}{\mathbb{R}}

\newcommand{\Z}{\mathbb{Z}}

\newcommand{\cir}{\mathbb{S}^1}

\DeclareMathOperator{\D}{\mathcal{D}}
\DeclareMathOperator{\diff}{\operatorname{Diff}}

\DeclareMathOperator{\Imm}{Imm}

\DeclareMathOperator{\dist}{dist}
\DeclareMathOperator{\diam}{diam}
\DeclareMathOperator{\Diff}{Diff}

\newcommand{\red}[1]{{\color{red} #1}}
\newcommand{\blue}[1]{{\color{blue} #1}}
\usepackage{pifont}%
\newcommand{\vmark}{\ding{51}}%
\newcommand{\xmark}{\ding{55}}%

\begin{document}

\title[Fractional Sobolev metrics on spaces of  curves]{
Completeness and geodesic distance properties for fractional Sobolev metrics on spaces of immersed curves}

\author{Martin Bauer}
\address{Department of Mathematics, Florida State University, Tallahassee, United States.}
\email{bauer@math.fsu.edu}
\author{Patrick Heslin}
\address{Department of Mathematics and Statistics, National University of Ireland, Maynooth, Kildare, Ireland.}
\email{patrick.heslin@mu.ie}
\author{Cy Maor}
\address{Einstein Institute of Mathematics, Hebrew University of Jerusalem, Israel.}
\email{cy.maor@mail.huji.ac.il}

\subjclass[2020]{58B20, 58D10, 35G55, 35A01}

\date{December 2023}

\keywords{infinite-dimensional Riemannian geometry, immersions, geodesic distance, completeness, global well-posedness, fractional Sobolev space}

\begin{abstract}
We investigate the geometry of the space of immersed closed curves equipped with reparametrization-invariant Riemannian metrics;
the metrics we consider are Sobolev metrics of possible fractional order $q\in [0,\infty)$. We establish the critical Sobolev index on the metric for several key geometric properties. Our first main result shows that the Riemannian metric induces a metric space structure if and only if $q>1/2$.
Our second main result shows that the metric is geodesically-complete (i.e., the geodesic equation is globally well-posed) if $q>3/2$, whereas if $q<3/2$ then finite-time blowup may occur.
The geodesic-completeness for $q>3/2$ is obtained by proving metric-completeness of the space of $H^q$-immersed curves with the distance induced by the Riemannian metric.

\end{abstract}

\maketitle

\setcounter{tocdepth}{1}
\tableofcontents

\section{Introduction}
\subsection*{Background and Motivation.}
Reparametrization-invariant Sobolev metrics on the space of immersed curves have been of central interest in recent years:
from an application point of view, they take a central role in the area of mathematical shape analysis, see e.g.~\cite{srivastava2016functional,younes2010shapes,bauer2014overview} and the references therein. These metrics also arise in higher-order gradient flows for various functionals \cite{schrader2023h,okabe2021convergence}.
From a theoretical point of view, they are the natural generalization of right-invariant Sobolev metrics on the diffeomorphism group of $\cir$; whose geodesic equations reduce to many important PDEs from hydrodynamics, including the Burgers, Camassa-Holm and Hunter-Saxton equations. For a comprehensive list of examples, see the book of Arnold and Khesin \cite{arnold2021topological}.

More recently, there has been an interest to extend the study of reparamerization-invariant Sobolev metrics to those of fractional order. 
This can be motivated, e.g., by applications in shape optimization in geometric knot theory~\cite{reiter2021sobolev,knappmann2023speed}; there, a main tool is using gradient-based approach for $H^{3/2}$ and $H^{3/2+\epsilon}$-type metrics, an exponent we show in this work to be critical for the completeness of the metric.
Fractional order metrics have already been investigated in the context of the aforementioned geometric approach to hydrodynamics. 
Well-known PDEs, including the Surface Quasi-Geostrophic equations \cite{washabaugh2016sqg} and the modified Constantin-Lax-Majda equation~\cite{wunsch2010geodesic}, arise as reduced geodesic equations for right-invariant Sobolev metrics of fractional order on diffeomorphism groups.

The geometry of infinite dimensional Riemannian manifolds is subtle and susceptible to pathologies and many elementary facts from finite dimensional geometry do not necessarily carry over. Indeed a smooth exponential map may not exist \cite{constantin2002geometric}, the geodesic distance can be degenerate or even vanish identically \cite{eliashberg1993bi,michor2005vanishing,magnani2020remark,jerrard2019geodesic}, and almost all statements of the classical Hopf-Rinow theorem fail to hold \cite{grossman1965, mcalpin1965, atkin1975}.

In the context of reparametrization-invariant Sobolev metrics on spaces of immersions, generally speaking, the higher the order of the Sobolev metric, the better behaved the Riemannian structure. 
The goal of this current paper is to identify the exact thresholds in which transitions between ``bad'' and ``good'' behaviors occur, as described below, and are summarized in Table~\ref{tab:Imm_vs_Diff}. 

\subsection*{Main results.}
We will now describe the main results of the present article. 
Our central object of interest is the space $\Imm(\cir,\R^d)$ of smooth immersions of closed curves in $\R^d$, endowed with a reparametrization-invariant Sobolev metric of order $q\in [0,\infty)$, denoted by $G^q$.
Each result is restated later in greater detail and generality, and includes also immersions of Sobolev (rather than smooth) regularity. 
Exact definitions of the spaces and the metrics considered here are given in Section~\ref{sec:Riem_geo}.

Our first main result concerns the induced geodesic distance:
Since the space of immersed curves $\Imm(\cir,\R^d)$ is an infinite dimensional manifold, the induced geodesic distance of a Riemannian metric is a-priori only guaranteed to be a semi-metric, i.e., distinct elements can be of zero distance~\cite{michor2005vanishing,eliashberg1993bi}.
The following result characterizes precisely for which metrics this occurs:
\begin{theorem*}[Geodesic distance]
    The geodesic distance of the reparametrization-invariant Sobolev metric of order $q\in [0,\infty)$, on the space of smooth immersed closed curves $\Imm(\cir,\R^d)$, induced a metric space structure if and only if $q>1/2$.
\end{theorem*}
A more detailed version of this Theorem is given in Theorem~\ref{non-vanishing}.
The fact that for $q=0$ the geodesic distance collapses was obtained by Michor--Mumford \cite{michor2006riemannian}, who also showed that the geodesic distance is not degenerate on the quotient shape space for $q\ge 1$.
Using different methods we extend their results in both directions.

Our second main result concerns the well-posedness of the corresponding geodesic equation: Bauer--Bruveris--Kolev~\cite{bauer2018fractional}  showed that these equations are locally well-posed when the order of the metric is at least $1$. 
Here we determine the critical index for global existence, i.e,
geodesic completeness of the metric:

\begin{theorem*}[Geodesic completeness]
    The reparametrization-invariant Sobolev metric $G^q$ on the space  of smooth immersed closed curves $\Imm(\cir,\R^d)$ is geodesically-complete if $q>3/2$ and is not if $q<3/2$.
\end{theorem*}
A more detailed version of this theorem, including for results on geodesic convexity and metric completeness for curves of Sobolev regularity, is given in Theorem~\ref{main theorem}. 
Together with the local well-posedness result for the geodesic equation \cite{bauer2018fractional}, our result implies that the corresponding geodesic equation is globally well-posed for $q>3/2$.
This was previously known only for integer order metrics $q\ge 2$ \cite{bruveris2015completeness}.

The proof of this result extends a method previously used for proving completeness of integer-order metrics \cite{bruveris2015completeness,bauer2020sobolev}: 
first proving that for $q>3/2$ the space of $H^q$-Sobolev immersions, endowed with our reparametrization-invariant Sobolev metric of order $q$, is \emph{metrically}-complete. 
Since this metric is a strong metric on this space, the geodesic completeness then follows by the only part of the Hopf-Rinow theorem that holds in infinite dimensions. 
The geodesic completeness in the smooth category is shown by an Ebin-Marsden-type no-loss-no-gain argument~\cite{ebin1970}. 
The main challenges here are the more complicated estimates that arise due to the fractional order norms. 
The proof for the geodesic incompleteness for $q<3/2$ follows from a simple example of shrinking circles.

\begin{table}[]
    \centering
    \begin{tabular}{c|c|c|c|c|c|c|c|c}
        $q$ (order of the metric) & $0$ & $(0,\frac{1}{2})$ & $\frac{1}{2}$ & $(\frac{1}{2},1)$ & $1$ & $(1,\frac{3}{2})$ & $\frac{3}{2}$ & $>\frac{3}{2}$ \\
        \hline
        smoothness\footnotemark[1]  & \blue{\xmark} \red{\xmark}  & \blue{\xmark} \red{\xmark} & \blue{\vmark} \red{\xmark} & \blue{\vmark} \red{\xmark} &
        \blue{\vmark} \red{\vmark}  & \blue{\vmark}  \red{\vmark}  & \blue{\vmark}  
        \red{\vmark}  & \blue{\vmark}  \red{\vmark} \\        
        metric space\footnotemark[2]  & \blue{\xmark}  \red{\xmark} & 
        \blue{\xmark}  \red{\xmark} & \blue{\xmark}  \red{\xmark} & \blue{\vmark}  \red{\vmark} & \blue{\vmark}  \red{\vmark} & \blue{\vmark}  \red{\vmark} & \blue{\vmark}  \red{\vmark} & \blue{\vmark}  \red{\vmark}\\
        geodesic\ completeness\footnotemark[3] & \blue{\xmark} \red{\xmark} & 
        \blue{?} \red{\xmark} & \blue{\xmark} \red{\xmark} & \blue{?} \red{\xmark} & 
        \blue{\xmark} \red{\xmark} & \blue{?} \red{\xmark} &
        \blue{\vmark} \red{?} & \blue{\vmark} \red{\vmark}\\
    \end{tabular}
    \vspace{0.2cm}
    \caption{Geometric properties of right-invariant $H^q$ metrics on $\Diff(\cir)$ (\blue{blue}) and $\Imm(\cir,\R^d)$ (\red{red}). Smoothness refers to whether the geodesic spray extends to a smooth vector field on the space of Sobolev diffeomorphisms (curves, resp.), and implies local well-posedness of the geodesic equation in both the Sobolev and smooth categories. Together with geodesic completeness this implies global well-posedness. 
    The contribution of this paper is virtually to all the results in red in the last two lines.}
    \label{tab:Imm_vs_Diff}
\end{table}
\footnotetext[1]{For the results concerning smoothness of the extended spray for $\Diff(\cir)$ we refer to \cite{constantin2002geometric,escher2011geometry} and for $\Imm(\cir,\mathbb R^d)$ to \cite{bauer2018fractional} (this is not an extensive list, and also the ones below are not).}
\footnotetext[2]{For the results  concerning the metric space structure of $\Diff(\cir)$ we refer to \cite{bauer2013geodesic,jerrard2019geodesic} and for the previously known cases on $\Imm(\cir,\mathbb R^d)$ to \cite{michor2006riemannian}.}
\footnotetext[3]{For the results concerning the geodesic completeness on $\Diff(\cir)$ we refer to \cite{escherkolev2014,camassa1993integrable,constantin1998wave,bauer2016geometric, preston2018euler,bauer2020geodesic} and for the previously known cases on $\Imm(\cir,\mathbb R^d)$ to \cite{bruveris2014geodesic,bruveris2015completeness}.}

\subsection*{Comparisons with results for right-invariant metrics on diffeomorphism groups.}
The results of this paper are analogous to those on Sobolev metrics on the group of diffeomorphism of the circle $\Diff(\cir)$; in particular, the critical exponents for completeness and vanishing-distance turn out to be the same, see Table~\ref{tab:Imm_vs_Diff}. 

While the eventual results are mostly similar for $\Diff(\cir)$ and $\Imm(\cir,\R^d)$, obtaining them for $\Imm(\cir,\R^d)$ is generally significantly harder. 
The reason for this is that $\Imm(\cir,\R^d)$ is a much richer space --- one can think of $\Diff(\cir)$ as a subspace of $\Imm(\cir,\R^d)$ of curves with a fixed image. 
From another perspective, results on $\Diff(\cir)$ can be reduced to arguments over the Lie algebra $\mathfrak{X}(\cir)$ of vector fields, whereas there is no equivalent for this on $\Imm(\cir,\R^d)$.
This can be seen, for example, in the proof of non-vanishing distance of the geodesic distance for $q>1/2$ (Theorem~\ref{non-vanishing}):
In the case of $\Diff(\cir)$, this is a simple application of the embedding $L^\infty\subset H^q$ \cite{bauer2013geodesic}.
However, using the same embedding for paths in $\Imm(\cir,\R^d)$, we obtain weights, depending on the length of the curves in the path (whereas in the case of $\Diff(\cir)$ the length is fixed); these lengths are not controlled from below when $q\in (1/2,3/2]$, and thus a more convoluted argument is needed.

Similarly the global well-posedness for $q>\frac{3}{2}$ on $\Diff(\cir)$ follows directly by abstract arguments~\cite{bruveris2017completeness}, using the fact that the $G^q$-metric extends to a strong, right invariant metric on the space of Sobolev diffeomorphisms of regularity $q$.  Similar arguments can be used to show that the $G^q$-metric induces a strong, reparametrization-invariant metric on the space of curves of Sobolev regularity $q$. Due to the more intricate nature of $\Imm(\cir,\mathbb R^d)$ this result cannot be used to directly conclude global well-posedness. Instead one has to carefully bound the dependence of several geometric quantities on $G^q$-metric balls and use this to prove metric completeness by direct estimates. 

\subsection*{Future directions}
The results of this paper also work for scale-invariant versions of the metrics $G^q$, and other length weights; it would be interesting to find optimal (or nearly optimal) conditions on the length-weights for which completeness holds for $q>3/2$, in the spirit of \cite{bruveris2017completeness2}.
We expect  our results to also extend to the case of manifold-valued curves, by combining the techniques of this paper with those of \cite{bauer2020sobolev}.
Regarding vanishing geodesic distance, it is still open whether, for $q<1/2$, the geodesic distance collapses completely, and whether it collapses also on shape space, which both hold in the case $q=0$ \cite{michor2006riemannian}.
Finally, we note that the geodesic completeness for the critical index, i.e., for the $G_c^{3/2}$ metric on $\Imm(\cir,\R^d)$, is still open. The completeness for the corresponding metric on $\Diff(\cir)$ was recently established \cite{bauer2020geodesic}.

\subsection*{Structure of the paper}
In Section~\ref{sec:prelim} we provide the necessary background, including definitions of the fractional Sobolev norms we use in this paper, the space of Sobolev immersions, and the reparametrization-invariant metrics. We also prove some useful inequalities, both on the Sobolev norms and on the Riemannian metrics. In Section~\ref{sec:geodesic_distance} we restate and prove the results regarding the geodesic distance. 
In Section~\ref{sec:completeness} we restate and prove the results regarding completeness properties. 
The appendix contains proofs of some results for fractional Sobolev spaces used throughout the paper.

\subsection*{Acknowledgements}
CM and MB were partially funded by BSF grant \#2022076. MB was partially funded by NSF grant DMS-1953244 and by the Austrian Science Fund grant P 35813-N. CM was partially funded by ISF grant 1269/19. PH was supported by the National University of Ireland's Dr. {\'E}amon de Valera Postdoctoral Fellowship.

\section{Preliminaries}\label{sec:prelim}

\subsection{Fractional Sobolev Spaces}\label{interpolation estimates}
Here we record some definitions and estimates pertaining to fractional Sobolev spaces. 
Our presentation follows closely that of Escher and Kolev \cite{escherkolev2014}. 

Throughout this paper, we identify $\cir = \R/\mathbb{Z}$, and let $\theta\in [0,1]$ be a parametrization of $\cir$ (with $0\sim 1$).
The fractional Sobolev space $H^q(\cir, \R^d)$ for $q \in \R$ is acquired by completing the space of smooth functions $C^\infty(\cir, \R^d)$ under the norm
\begin{equation}\label{sobolev norm}
    \norm{f}_{H^q}^2 = \sum_{n \in \Z^d} (1+n^2)^q \sabs{\hat{f}(n)}^2,
\end{equation}
where $\displaystyle \hat{f}(n)$ denotes the Fourier transform of $f$.  
We recall the Sobolev embedding theorem for fractional spaces: 
\begin{proposition}[Sobolev Embedding Theorem]
    For $q>\frac{1}{2} + k$ the space $H^q(\cir, \R^d)$ continuously embeds into the classical space $C^k(\cir, \R^d)$ of $k$-times continuously differentiable functions.
\end{proposition}
Proofs of this statement can be found in many standard references, e.g., \cite[Section~2.7.1]{triebel1983theory}. We further define the space $H^q(\cir, \cir)$ to consist of all self-maps of the circle which, when composed with any chart, are in $H^q(\cir, \R)$. 
If we require that $q>\frac{3}{2}$, it follows from the Sobolev Embedding Theorem that $H^q(\cir, \cir) \hookrightarrow C^1(\cir, \cir)$. 
Hence, by the Inverse Function Theorem, we may define the space of $H^q$-diffeomorphisms of the circle as $$\D^q(\cir)=H^q(\cir, \cir) \cap \{C^1-\text{diffeomorphisms of } \cir\}.$$ This space is an infinite-dimensional Hilbert manifold modeled on $H^q(\cir, \R)$ c.f., \cite{ebin1970}.  
It is in addition a half-Lie group~\cite{marquis2018half,bauer2023regularity}, i.e., a topological group under composition, where for any $\varphi \in \D^q(\cir)$ right translation
\begin{align*}
    \D^q(\cir) \rightarrow \D^q(\cir) \ ; \ \eta \mapsto \eta \circ \varphi
\end{align*}
is smooth, but left translation
\begin{align*}
    \D^q(\cir) \rightarrow \D^q(\cir) \ ; \ \eta \mapsto \varphi \circ \eta
\end{align*}
is only continuous. It further acts on $H^q(\cir, \R)$ where, again for fixed $\varphi \in \D^q(\cir)$, the following map is smooth
\begin{align*}
    H^q(\cir, \R^d) \rightarrow H^q(\cir, \R^d) \ ; \ f \mapsto f \circ \varphi.
\end{align*}

Throughout our arguments we will require estimates on products and compositions with respect to the \textit{homogeneous} Sobolev seminorm defined by
\begin{align}\label{hom_sobolev norm}
    \norm{f}_{\dot{H}^q}^2 = \sum_{n \in \Z} n^{2q} \sabs{\hat{f}(n)}^2.
\end{align}
Finally, we note that this seminorm can be rewritten as
\[
\norm{f}^{2}_{\dot{H}^q} = \int_{\mathbb{S}^1} \ip{\Lambda^{2q} h,h} \,d\theta,
\]
where $\Lambda := H \partial_\theta$ is the pseudo-differential operator with symbol $\widehat{\Lambda}(m)=\abs{m}$.

Our central estimates are as follows. From a notational standpoint, we will write $\norm{\cdot}_1 \simeq_{R} \norm{\cdot}_2$, $\norm{\cdot}_1 \lesssim_{R} \norm{\cdot}_2$, etc. to indicate an equivalence or inequality is valid up to a constant depending continuously on $R$.

\begin{lemma}\label{lemma:com-pro} Consider the Sobolev spaces and norms as defined above.
\begin{enumerate}[(i)]
\item For $0 < a \leq b$ we have
    \begin{equation}\label{non-invariant nesting}
        \norm{f}_{\dot{H}^{a}} \leq \norm{f}_{\dot{H}^{b}},
    \end{equation}
    for all $f \in H^b(\cir,\R^d)$. \\
    \item For $b>\frac{1}{2}$ and $0\leq a \leq b$, we have the following estimate on products
    \begin{equation}\label{full norm product}
        \norm{f \cdot g}_{{H}^a} \lesssim_{(a,b)} \norm{f}_{{H}^a}\norm{g}_{{H}^b},
    \end{equation}
    for all $f \in H^a(\cir, \R^d)$ and $g \in H^b(\cir,\R^d)$. \\
    \item For $b>\frac{1}{2}$ and $0\leq a \leq b$, we have the following estimate on products for the homogeneous norm
    \begin{equation}\label{product}
        \norm{f \cdot g}_{\dot{H}^a} \lesssim_{(a,b)} |\hat{f}(0)| \norm{g}_{\dot{H}^a} + \abs{\hat{g}(0)}\norm{f}_{\dot{H}^a}+\norm{f}_{\dot{H}^a}\norm{g}_{\dot{H}^b},
    \end{equation}
    for all $f \in H^a(\cir, \R^d)$ and $g \in H^b(\cir,\R^d)$. \\ 
    \item For $0\leq a \leq 1$ we have the following estimate on products for the homogeneous norm
    \begin{equation}\label{unit vector product}
        \norm{f\cdot g}_{\dot{H}^a} \lesssim_{a} \norm{f}_{\dot{H}^a} \norm{g}_{L^\infty}  + \norm{f}_{L^\infty} \norm{g}_{\dot{H}^a},
    \end{equation}
    for all $f,g \in H^a(\cir, \R^d)$. \\
    \item For $b>\frac{3}{2}$ and $0 \leq a \leq 1$, we have the following estimate on compositions
    \begin{equation}\label{composition}
        \norm{f \circ \varphi}_{\dot{H}^a} \leq \norm{(\varphi^{-1})_\theta}_{L^\infty}^{\tfrac{1-a}{2}} \norm{\varphi_\theta}_{L^\infty}^{\tfrac{a}{2}} \norm{f}_{\dot{H}^a},
    \end{equation}
    for all $f \in H^a(\cir, \R^d)$ and $\varphi \in \D^b(\cir)$.
\end{enumerate}
\end{lemma} 

Details of the proofs for these estimates can be found in Appendix \ref{appendix}.

\subsection{Riemannian Geometry of Immersed Curves}\label{sec:Riem_geo}

Here we introduce the setting for the results contained in this paper. Further details of the constructions can be found in Michor--Mumford \cite{michor2006riemannian}. 

We consider the space of smooth immersions of $\cir$ into Euclidean space
\begin{equation}
   \Imm(\cir,\R^d) := \{ c \in C^\infty(\cir,\R^d) \ \vert \ \abs{c_\theta} \neq 0 \}.
\end{equation}
It is an open subset of $C^\infty(\cir,\R^d)$ and hence inherits the structure of an infinite-dimensional Frech{\'e}t manifold with tangent space at the point $c \in \Imm(\cir,\R^d)$ given by
\begin{equation}
    T_c\Imm(\cir,\R^d) : = \{h \in C^\infty (\cir,T\R^d) \ \vert \ \pi_{T\R^d} \circ h = c\}.
\end{equation}
Next, we define the space of smooth, orientation preserving diffeomorphisms of $\cir$
\begin{equation}
    \diff(\cir) := \{\varphi\in C^\infty(\cir,\cir) \ \vert \ \varphi \text{ is a bijection}\}.
\end{equation}
This is an infinite-dimensional Frech{\'e}t Lie group which acts on $\Imm(\cir,\R^d)$ on the right by composition
\begin{equation}
    \Imm(\cir,\R^d) \times \diff(\cir) \rightarrow \Imm(\cir,\R^d) \ ; \ (c,\varphi) \mapsto c \circ \varphi.
\end{equation}

We are interested in equipping $\Imm(\cir,\R^d)$ with reparametrization-invariant Riemannian metrics $G$,  i.e., for all $c \in \Imm(\cir,\R^d), h,k \in T_f\Imm(\cir,\R^d)$ and $\varphi \in \diff(\cir)$
\begin{equation}
    G_{c\circ\varphi}\left(h\circ\varphi, k\circ\varphi\right) = G_c(h,k).
\end{equation}
The importance of these metrics stems from the fact that they descend to Riemannian metrics on the \textit{shape space}
\begin{equation}
    \text{B}_i(\cir,\R^d):= \Imm(\cir,\R^d) / \diff(\cir),
\end{equation}
which carries the structure of an infinite-dimensional orbifold. 

A subclass of reparametrization-invariant metrics are Sobolev-type metrics of order $q \in \R$:
\begin{equation}\label{invariant metric}
    G_c^q(h,k) = \int_{\mathbb{S}^1} \ip{L_c^q h,k} \ ds
\end{equation}
where $L^q_c$ is a self-adjoint, invertible, pseudo-differential operator of order $2q$ depending on $c$ in such a way as to keep $G$ reparametrization-invariant and $ds = \abs{c_\theta} d\theta$ denotes integration with respect to arc length. It has been observed in~\cite{bauer2018fractional}, that the requirement that the metric \eqref{invariant metric} be invariant under $\diff(\cir)$ tells us that $\ip{\cdot, \cdot}_{G_c^q}$ is completely determined by its behavior on constant speed curves. In particular, if we define the constant speed reparametrization for $c$:
\begin{equation}
    \psi_c (\theta) = \frac{1}{l_c}\int_0^{\theta} \vert c_\theta(\sigma) \vert \ d\sigma, 
\end{equation}
where $l_c$ denotes the length of $c$, we have $\vert (c \circ \psi^{-1}_c)_\theta\vert = l_c$; which in turn gives us that $L^q_c = R_{\psi_c} L^q_{c \circ \psi_c^{-1}} R_{\psi_c^{-1}}$. We now assume $L^q_c$ has the form:
\begin{equation}
    L^q_c = R_{\psi_c} \left( 1 + \left(\frac{1}{l_c}\right)^{2q}\Lambda^{2q} \right) R_{\psi_c^{-1}},
\end{equation}
where, as before, $\Lambda := H \partial_\theta$ is the pseudo-differential operator with symbol $\widehat{\Lambda}(m)=\abs{m}$. In summary, our full and our homogeneous reparametrization-invariant metrics of interest are, respectively, given by:
\begin{equation}\label{main metric}
    G_c^q(h,k) := \int_{\mathbb{S}^1} \ip{R_{\psi_c} \left( 1 + \left(\frac{1}{l_c}\right)^{2q}\Lambda^{2q} \right)R_{\psi_c^{-1}} h,k} \ ds
\end{equation}
and 
\begin{equation}\label{main homogeneous metric}
    \dot{G}_c^q(h,k) := \int_{\mathbb{S}^1} \ip{R_{\psi_c} \left(\frac{1}{l_c}\right)^{2q}\Lambda^{2q}R_{\psi_c^{-1}} h,k} \ ds.
\end{equation}

We further denote differentiation with respect to arc length by $D_s = \frac{1}{\abs{c_\theta}} \partial_\theta$. It is not difficult to show that $\norm{D_s h}_{\dot{G}_c^q} = \norm{h}_{\dot{G}_c^{q+1}}$. For constant speed curves this follows from integration by parts. 
The result for non-constant speed curves then follows from the reparametrization-invariance of $\dot{G}_c^q$.

We also consider alongside the above space $\Imm(\cir, \R^d)$, its finite smoothness counterparts. 
For ${r}>\frac{3}{2}$ we have the Sobolev completions
\begin{align}
   \mathcal{I}^{r}(\cir,\R^d) := \{ c \in H^{r}(\cir,\R^d) \ \vert \ \abs{c_\theta} \neq 0 \} \ ,\\ T_c\mathcal{I}^{r}(\cir,\R^d) : = \{h \in H^{r}(\cir,T\R^d) \ \vert \ \pi_{T\R^d} \circ h = c\}.
\end{align}

The following result is due to Bauer et al. \cite[Theorems 6.4 and 7.1]{bauer2018fractional}.
\begin{proposition}
    {For $r>\frac{3}{2}$ and $q\in [\frac{1}{2},\frac{r}{2}]$ the metric $G_c^q$ as in \eqref{main metric} is a smooth metric on $\mathcal{I}^{r}(\cir, \R^d)$.
    If $q\ge 1$, then it induces a smooth exponential map on $\mathcal{I}^{r}(\cir, \R^d)$, which is a local diffeomorphism.
    If $q=r$, then the metric extends to a strong metric on $\mathcal{I}^{r}(\cir, \R^d)$, i.e., it induces on each tangent space the original $H^r$-topology.}
\end{proposition}

We now express the induced homogeneous reparametrization-invariant norm in terms of the usual homogeneous Sobolev norm.

\begin{lemma}\label{relating homogeneous norms}For $r>\frac{3}{2}$, $c \in \mathcal{I}^r$ and $h \in T_c \mathcal{I}^r$ and $q\leq r$ we have
    \begin{equation}\label{invariant versus non-invariant}
        \norm{h}_{\dot{G}_c^q}^2 = l_c^{1-2q}\norm{h \circ \psi_c^{-1}}_{\dot{H}^q}^2.
    \end{equation}
    Moreover, the subsequent inequality holds {for any $q_1\le q_2\le r$}
    \begin{equation}\label{invariant nesting}
        \norm{h}_{\dot{G}_c^{{q_1}}} \leq l_c^{{q_2-q_1}}\norm{h}_{\dot{G}_c^{q_2}}.
    \end{equation}
    Lastly, when $1 \leq q \leq r$ we have
    \begin{equation}\label{homogeneous versus full norm nesting}
        \norm{h}_{\dot{G}_c^1} \leq \norm{h}_{{G}_c^q}.
    \end{equation}
\end{lemma}
\begin{proof} Applying a change of coordinates we compute directly
    \begin{align*}
        \norm{h}_{\dot{G}_c^q}^2 &= \int_{\cir}\ip{\left(\frac{1}{l_c}\right)^{2q}\Lambda^{2q} R_{\psi_c^{-1}} h, R_{\psi_c^{-1}} h} \ l_c \ d\theta \\
        &= l_c^{1-2q} \int_{\cir}\ip{\Lambda^{2q}R_{\psi_c^{-1}}h, R_{\psi_c^{-1}}h} \ d\theta \\
        &= l_c^{1-2q}\norm{h \circ \psi_c^{-1}}_{\dot{H}^q}^2.
    \end{align*}
    
    The inequality \eqref{invariant nesting} follows immediately from the above and \eqref{non-invariant nesting}.

    Finally, note that, as the metrics \eqref{main metric} and \eqref{main homogeneous metric} are invariant under reparametrization, it will suffice to establish the inequality \eqref{homogeneous versus full norm nesting} on constant speed curves, i.e., where $\psi_c(\theta)=\theta$. By \eqref{invariant versus non-invariant} we acquire that
    \begin{align*}
        \norm{h}_{\dot{G}_c^1}^2 &= l_c^{-1}\norm{h}_{\dot{H}^1}^2 = l_c^{-1} \sum_{j\in \Z^2} \abs{j}^2 \abs{\hat{h}(j)}^2
    \end{align*}
    and
    \begin{align*}
        \norm{h}_{G_c^q}^2 &= \norm{h}_{G_c^0}^2 + \norm{h}_{\dot{G}_c^q}^2 = l_c \norm{h}_{L^2}^2 + l_c^{1-2q}\norm{h}_{\dot{H}^q}^2 \\
        &= l_c \sum_{j\in \Z^2} \abs{\hat{h}(j)}^2 + l_c^{1-2q} \sum_{j\in \Z^2} \abs{j}^{2q} \abs{\hat{h}(j)}^2 \\
        &= \sum_{j \in \Z^2} \left( l_c + l_c^{1-2q}\abs{j}^{2q}\right) \abs{\hat{h}(j)}^2.
    \end{align*}
    Hence, \eqref{homogeneous versus full norm nesting} will hold if, for any $l_c>0$ and $\abs{j} \in \Z^2$, we have
    \begin{align*}
        l_c^{-1} \abs{j}^2 \leq l_c + l_c^{1-2q}\abs{j}^{2q}.
    \end{align*}
    For $\abs{j} = 0$, this is immediate. For $\abs{j}\neq0$ we divide both sides by $l_c^{-1}\abs{j}^2$ and obtain
    \begin{align*}
        1 \leq \left(\frac{l_c}{\abs{j}}\right)^2 + \left(\frac{\abs{j}}{l_c}\right)^{2(q-1)},
    \end{align*}
    which holds for any $q\geq 1$.
    \end{proof}

The next lemma, which establishes a Sobolev Embedding-type theorem for our repara\-metrisation-invariant norms,  extends a result for first order metrics contained in \cite[Lemma~2.14]{bruveris2014geodesic} to fractional orders. 
The important point here is the explicit dependence of the embedding constant on the length of the underlying curve.
\begin{lemma}\label{sobolev embedding} 
If $r>\frac{3}{2}$ and $\frac{1}{2} < q \leq r$, then there exists a constant $C=C(q,d)>0$ such that, for all $h \in T_c \mathcal{I}^r$ and all $\ell \in(0,l_c]$, we have
\begin{equation}\label{eq:Sob_embed}
    \norm{h}_{L^{\infty}(\mathbb{S}^1)} \leq C \sqrt{\frac{1}{\ell} \left(\norm{h}_{G_c^0}^2 +\ell^{2q} 
    \norm{h}_{\dot{G}_c^q}^2\right)}
    {\le C \max\{\ell^{-\frac{1}{2}}, \ell^{q-\frac{1}{2}}\} \norm{h}_{G_c^q}}.
\end{equation}
\end{lemma}
\begin{proof}
This inequality, for standard Sobolev spaces, i.e., with $H^q$ instead of $G^q$, is well-known: 
For integer order $q\ge 1$ this result appears in \cite[Theorem 7.40]{leoni2017first}, and for $q\in (1/2,1)$ it appears in \cite[Theorem 2.8]{leoni2023first}. 
The general case follows by a combination of these results. 
Our case reduces to this case by reparametrization, as shown below for the case $\ell=l_c$.

By the usual Sobolev Embedding Theorem, there exists $C=C(q,d) > 0$ such that $\norm{h}_{L^{\infty}(\mathbb{S}^1)} \leq C \norm{h}_{H^q(d\theta)}$. Hence, we have
\begin{align*}
    \norm{h}_{L^{\infty}(\mathbb{S}^1)} &= \norm{h\circ\psi_c^{-1}}_{L^{\infty}(\mathbb{S}^1)} \\
    & \leq C \norm{h\circ\psi_c^{-1}}_{H^q(d\theta)} \\
    & = C \sqrt{\frac{1}{l_c} \left(\norm{h}_{G_c^0}^2 + \left(\frac{1}{l_c}\right)^{-2q}\norm{h}_{\dot{G}_c^q}^2\right)}.\qedhere
\end{align*}
\end{proof}

\section{Geodesic Distance}\label{sec:geodesic_distance}
In this section we study the induced geodesic distance of our class of metrics.
Recall that any Riemannian metric induces a geodesic distance defined as the infimum over the length of all differentiable paths with fixed end points. 
As mentioned in the introduction, in finite dimensions this will always induce a metric space structure, however, in infinite dimensions this is not necessarily the case. 
We say that an induced geodesic distance is \textit{degenerate} if there exists a pair of points for which we can find an arbitrarily short path connecting them, i.e., the geodesic distance between the points is zero.

Initial investigations into geodesic distance in the context of spaces of immersed curves are due to Michor and Mumford \cite{michor2006riemannian,bauer2012vanishing}. Their results show that the geodesic distance is degenerate if $q=0$, but that it is a true distance on the quotient shape space $B_i(\mathbb{S}^1,\R^d)=\Imm(\mathbb{S}^1,\R^d)/\operatorname{Diff}(\mathbb{S}^1)$ if $q\geq 1$. Similar non-degeneracy results can be obtained on the whole space for $q\ge 1$ using the square-root transform~\cite{srivastava2010shape}. These results naturally raise the question, for which $q\in (0,1)$ this change of behavior occurs. Our main result of this section provides an answer to this question.
\begin{theorem}\label{non-vanishing}
The geodesic distance of the reparametrization invariant Sobolev metric $G_c^q$, denoted $\dist_{G_c^q}$, is non-degenerate if and only if $q>\frac{1}{2}$. 
More precisely,
\begin{enumerate}[(i)]
\item For any $q\leq \frac12$ there exists distinct curves  $c_0\neq c_1\in \Imm(\mathbb{S}^1,\R^d)$ such that $\dist_{G_c^q}(c_0,c_1)=0$.
\item For any $q>\frac{1}{2}$ and any $c_0\neq c_1\in \Imm(\mathbb{S}^1,\R^d)$ the geodesic distance $\dist_{G_c^q}(c_0,c_1)$ is non-zero.
\end{enumerate}
Furthermore, if $q>\frac{1}{2}$ we obtain the bound
\[
\dist_{G_c^q}(c_0,c_1)\geq C \min\{\|c_0-c_1\|_\infty,\diam_{\text{max}}\}\min\{\diam_{\text{max}}^{1/2},1\},
\]
for some constant $C=C(q,n)>0$. Here $\diam(c_i)$ is the diameter of the image of the curve $c_i$ and $\diam_{\text{max}} = \max\{\diam(c_0),\diam(c_1)\}$.

Lastly, if $q\geq 1$ we obtain the additional estimate
\[
\dist_{G_c^q}(c_0,c_1)\geq \sqrt{\int_{\cir} \left|\frac{\partial_\theta c_0}{|\partial_\theta c_0|^{\tfrac12}}-\frac{\partial_\theta c_1}{|\partial_\theta c_1|^{\tfrac12}}\right|^2 d\theta}=
\sqrt{l_{c_0}+l_{c_1}-\int_{\cir} 
\frac{\langle \partial_\theta c_0,\partial_\theta c_1\rangle}{|\partial_\theta c_0|^{\tfrac12}|\partial_\theta c_1|^{\tfrac12}} d\theta.}
\]
These results continue to hold on the space of Sobolev immersions $\mathcal{I}^r(\cir,\R^d)$, as long as the metric $G_c^q$ is defined on it.
\end{theorem}

\begin{remark}
For $q=0$ it has been shown in~\cite{michor2006riemannian}, that the geodesic distance vanishes identically on all of $\Imm(\mathbb{S}^1,\R^d)$. This implies, in particular, also the degeneracy of the induced geodesic distance on the shape space $B_i(\mathbb{S}^1,\R^d)$. Our result regarding the degeneracy for $q\leq \frac{1}{2}$ is significantly weaker: we only show the existence of distinct immersions, such that their geodesic distance is zero and we do not prove that this holds for arbitrary immersions. Furthermore, our examples are of the type $c_1=c_0\circ \varphi_1$ with $\varphi_1\in \operatorname{Diff}(\mathbb{S}^1)$. These elements are, however, identified in shape space $B_i(\mathbb{S}^1,\R^d)$. Thus this result does not resolve the degeneracy on the quotient space, but only for the space of immersions. We believe that the index $q=\frac12$ is also critical for the geodesic distance on  $B_i(\mathbb{S}^1,\R^d)$, but the necessary estimates seem quite challenging and we leave this question open for future research.  
\end{remark}

To prove Theorem~\ref{non-vanishing} we will first collect a useful estimate pertaining to the diameter of the initial curve $c_0$. 
\begin{lemma}\label{diameter}
    Let $c_0,c_1\in \Imm(\mathbb{S}^1,\R^d)$.
    If $l_{c_1}\le \diam(c_0)$, then $$\frac{1}{4}\diam(c_0) \leq \|c_0-c_1\|_{\infty}.$$
\end{lemma}

\begin{proof}
    The diameter of any closed curve is at most half its length.
    Thus, by assumption,
    \[
    \diam(c_1) \le \frac{1}{2}\diam(c_0).
    \]
    Now, let $t,s\in \mathbb{S}^1$ such that $|c_0(t)-c_0(s)| = \diam(c_0)$, then
    \[
    \begin{split}
    \diam(c_0) &\le |c_0(t) - c_1(t)| + |c_1(t)-c_1(s)| + |c_1(s)-c_0(s)| \\
        &\le 2\|c_0-c_1\|_\infty + \diam(c_1) \le 2\|c_0-c_1\|_\infty + \frac{1}{2}\diam(c_0),
    \end{split}
    \]
    from which the claim follows.
\end{proof}

With this at hand, we now proceed to the proof of the main theorem for this section.
\begin{proof}[Proof of Theorem~\ref{non-vanishing}]
  We start with showing the non-degeneracy for $q>\frac12$. 
  Let $c(t)=c(t,\cdot)$ with $t\in[0,1]$ be any path between $c_0$ and $c_1$.
    Denote
    \[
    t_0 = \max\{t\in [0,1] ~:~ \diam(c_0) \leq l_{c(s)} \text{ for all } s \leq t\}.
    \]
    Note that $t_0>0$ since $l_{c_0} \ge 2 \diam(c_0)$.
    By definition, $l_{c(t)}\ge \diam(c_0)$ for all $t\in [0,t_0]$, and we can use \eqref{eq:Sob_embed} with $\ell=\min\{1,\diam(c_0)\}$ to obtain
    \[
    \|c_0-c(t_0)\|_{\infty} \le \int_0^{t_0} \|\partial_t c\|_\infty \,dt \le C\max\{\diam(c_0)^{-1/2},1\} \int_0^{t_0} \|\partial_t c\|_{G_c^q} \,dt.
    \]
    If $t_0=1$ we are done.
    Otherwise, $t_0<1$, and thus $\ell_{c(t_0)} = \diam(c_0)$.
    Therefore, by Lemma \ref{diameter}, we have
\begin{align}
    \frac{1}{4}\diam(c_0)& \le \|c_0-c(t_0)\|_{\infty} \\
    &\le C\max\{\diam(c_0)^{-1/2},1\}  \int_0^{t_0} \|\partial_t c\|_{G_c^q} \,dt\\
    &\le C\max\{\diam(c_0)^{-1/2},1\}  \int_0^{1} \|\partial_t c\|_{G_c^q} \,dt.
\end{align}
    As this estimate holds for any path connecting $c_0$ to $c_1$, it also holds for the infimum and thus we have obtained the desired bound for the geodesic distance.

    Next we prove the additional bound for $q\geq 1$. Therefore we first introduce the so-called SRV transform~\cite{srivastava2010shape}, which is given as the mapping
    \begin{equation}
     c\mapsto \frac{c_\theta}{|c_\theta|^{\tfrac12}}.
    \end{equation}
    It has been shown that the SRV transform~\cite{srivastava2010shape} is a Riemannian isometry from $\Imm(\cir,\mathbb R^d)$ equipped with the $H^1$-type metric 
    \begin{align}
        &G^{\operatorname{SRV}}_c(h,h)=\int_{\cir} \langle D_sh^\top,D_sk^{\top}\rangle +\frac14\langle D_sh^\bot,D_sk^{\bot}\rangle ds\\
        &\text{with }D_sh^\top=\langle D_s h, D_s c \rangle D_sh,\text{ and }  D_sh^\bot=D_sh-D_sh^\top
    \end{align}
    to a submanifold of the space of all smooth functions $C^{\infty}(\cir,\mathbb R^d)$ equipped with the flat $L^2$ (Riemannian) metric. It is easy to see that the $H^1$-metric $G_c^1$ is lower bounded by the SRV metric $G^{\operatorname{SRV}}$ and thus  the same is true for their geodesic distances. Finally we note, that the geodesic distance of a submanifold is bounded by the geodesic distance on the surrounding space and that the geodesic distance of the flat $L^2$-metric is simply given by the $L^2$-difference of the functions. Combining these observations leads to the desired lower bound. 
    
    It remains to prove the degeneracy for $q\leq \frac12$. To this end, consider any $c_0\in\Imm(\mathbb{S}^1,\mathbb R^d)$ and let $c_1=c_0\circ\varphi_1$ for some fixed reparametrization $\operatorname{id}\neq\varphi_1\in\operatorname{Diff}(\cir)$. We aim to show that the geodesic distance between $c_0$ and $c_1$ is zero. For the sake of simplicity we assume that $c_0$ is parametrized by arc length.
    We now consider  any path $\varphi:[0,1]\mapsto \diff(\cir)$ connecting $\operatorname{id}$ to $\varphi_1$. Then $c(t)=c_0(\varphi(t))$ is a path in the manifold of immersions that connects $c_0$ to $c_1$. From this we have 
    \begin{align}
        \dist_{G_c^q}(c_0,c_1)&\leq \int_0^{1} \|\partial_{t} c \|_{G_c^q} \,dt \\
        &\leq \int_0^{1}\|\left(\partial_{\theta} c_0 \circ \varphi \right)  \partial_t \varphi \|_{G_c^q} \,dt \\
        &= \int_0^{1}\| \partial_{\theta} c_0  \left(\partial_t \varphi \circ \varphi^{-1} \right) \|_{G_c^q} \,dt \\
        & \leq \int_0^{1} C(l_{c}) \|\partial_{\theta} c_0  \left(\partial_t \varphi \circ \varphi^{-1} \right)\|_{H^q} \,dt,
        \end{align}
    where we used Lemma~\ref{relating homogeneous norms} for the expression of the $G^q$ metric in the last step.
    Since the length of $c(t)$ is constant in time, i.e., $l_{c_0}=l_{c(t)}$ we can bound this via
    \begin{align}
        \dist_{G_c^q}(c_0,c_1)&\leq C(l_{c_0})
        \int_0^{1} \|\partial_{\theta} c_0  \left(\partial_t \varphi \circ \varphi^{-1} \right)\|_{H^q} \,dt\\&\leq
        C(l_{c_0})\int_0^{1} \|\partial_{\theta} c_0\|_{H^1} \|  \partial_t \varphi_t\circ\varphi_t^{-1}\|_{H^q} \,dt\\&=\tilde C(l_{c_0}, \|c_0\|_{H^2})\int_0^{1} \|  \partial_t \varphi_t\circ\varphi_t^{-1}\|_{H^q} \,dt\,,
    \end{align}
    where we used \eqref{full norm product} with $a=q$ and $b=1$ in the last inequality.
    Note that the norm on the right hand side is exactly the right invariant $H^q$-norm on $\operatorname{Diff}(\cir)$. Since this inequality holds for any  path connecting $\operatorname{id}$ to $\varphi_1$ in $\operatorname{Diff}(\cir)$ this implies
    \begin{align}
    \dist_{G_c^q}(c_0,c_1)&\leq C(c_0)\dist^{\operatorname{Diff}}_{H^q}(\operatorname{id},\varphi_1)
    \end{align}
    and thus we obtain the desired result since the geodesic distance of the right invariant $H^q$-metric on $\operatorname{Diff}(\cir)$ vanishes for every $q\leq\frac12$, cf.~\cite{bauer2013geodesic}. 
\end{proof}

 \section{Completeness Properties}\label{sec:completeness}
This section concerns the second central goal of this paper: the extension of the results of Bruveris, Michor and Mumford~\cite{bruveris2014geodesic,bruveris2015completeness} on   completeness of integer order Sobolev metrics to fractional orders. 
The main result of this section is the following:

\begin{theorem}\label{main theorem}
    If $r>\frac{3}{2}$ then
    \begin{enumerate}
        \item The space $\mathcal{I}^r(\mathbb{S}^1,\R^d)$ equipped with the geodesic distance $\dist_{G_c^r}$ induced by the metric $G_c^r$ given in \eqref{main metric} is metrically complete.
         \item The Riemannian manifold $\left(\mathcal{I}^r(\mathbb{S}^1,\R^d),G_c^r\right)$ is geodesically convex, i.e., for any pair of curves $c_1$ and $c_2\in\mathcal{I}^r(\mathbb{S}^1,\R^d)$ there exists a minimizing geodesic connecting them.
        \item If $\frac{3}{2} < q \leq r$, then the Riemannian manifold $\left(\mathcal{I}^r(\mathbb{S}^1, \R^d), G_c^q\right)$ is geodesically complete. This result continues to holds for $r=\infty$, i.e., the Fr\'echet manifold $\Imm(\cir,\R^d)$ equipped with the $G_c^q$-metric is geodesically complete.  
        \item If $\frac12<q<r$, then the space $\mathcal{I}^r(\mathbb{S}^1,\R^d)$ equipped with the geodesic distance $\dist_{G_c^q}$ induced by the metric $G_c^q$ given in \eqref{main metric} is metrically incomplete. If $q>\frac32$ then the corresponding metric completion is exactly $\mathcal{I}^q(\cir, \R^d)$. 
        \item If $q<\frac{3}{2}$ then the Riemannian manifold $\left(\mathcal{I}^r(\mathbb{S}^1, \R^d), G_c^q\right)$ is geodesically incomplete.
    \end{enumerate}
\end{theorem}
Note that fourth claim only discusses the case $q>1/2$ since, by Theorem~\ref{non-vanishing}, for $q\le 1/2$, the geodesic distance does not induce a metric space structure at all.

The bulk of the work in establishing this theorem lies in proving the first result. 
The second result then follows from an analogous argument to the integer order case. 
The proof of the third claim follows from the first by an Ebin-Marsden-type no-loss no-gain result. 
The fourth part is shown mainly via some rather soft arguments and finally the fifth is proven using an explicit example.

The proof of metric completeness hinges on the following Lemma, which establishes an equivalence of the $r$th-order invariant and non-invariant norms on $G_c^r$-metric balls.
\begin{lemma}\label{lemma 1}
    If $r>\frac{3}{2}$, then, for any $G_c^r$-metric ball $B_{G_c^r}(c_0, \rho)$ in $\mathcal{I}^r$, there exists a constant $\alpha(r,c_0,\rho)>0$ such that, for all $c \in B_{G_c^r}(c_0,\rho)$ and all $h \in T_c\mathcal{I}^r$, we have
    \begin{equation}\label{equivalence}
        \alpha^{-1}\norm{h}_{H^r} \leq \norm{h}_{G_c^r} \leq \alpha \norm{h}_{H^r}.
    \end{equation}
\end{lemma}

To establish this result we first recall a useful lemma for establishing boundedness on $G_c^r$-metric balls, cf. \cite[Lemma~3.2]{bruveris2015completeness}.
\begin{lemma}\label{gronwall lemma}
    Let $(X, \norm{\cdot}_X)$ be a normed space with $f: (\mathcal{I}^r, \dist_{G_c^r}) \rightarrow (X, \norm{\cdot}_X)$ a $C^1$ function. For any $c \in \mathcal{I}^r$ and $h \in T_c\mathcal{I}^r$ we denote the derivative of $f$ at $c$ in the direction $h$ by
    \begin{equation}\label{directional derivative}
        D_{c,h} (f) = \frac{d}{dt}\bigg\vert_{t=0} f(\sigma(t)),
    \end{equation}
    where $\sigma:(-\varepsilon,\varepsilon) \rightarrow \mathcal{I}^r$ is a $C^1$ path with $\sigma(0)=c$ and $\dot{\sigma}(0)=h$.
    
    Assume that for {every metric ball} $B_{G_c^r}(c_0,\rho)$ there exists a constant $\beta(r,c_0,\rho)>0$ such that
    \begin{equation}\label{linear gronwall estimate}
        \norm{D_{c,h} (f)}_X \leq \beta (1 + \norm{f(c)}_X) \norm{h}_{G_c^r} ,
    \end{equation}
    for all $c \in B_{G_c^r}(c_0,\rho)$ and $h \in T_c\mathcal{I}^r$. Then $f$ is Lipschitz continuous, and, in particular, bounded on every $G_c^r$-metric ball $B_{G_c^r}(c_0, \rho)$ in $\mathcal{I}^r$.
\end{lemma}
The above lemma also holds for vector spaces equipped with semi-norms. In particular we may consider mappings such as $f: (\mathcal{I}^r, \dist_{G_c^r}) \rightarrow (H^p, \norm{\cdot}_{\dot{H}^p})$.
    
Throughout this section, for notational simplicity, we will {denote $A \lesssim B$ when there exists a constant $c>0$, depending on $r$, $c_0$ and $\rho$, such that $A\le cB$ on a $G_c^r$-metric ball $B_{G_c^r}(c_0, \rho)$ in $\mathcal{I}^r$.
Similarly, we will write $A\simeq B$ if $A\lesssim B$ and $B\lesssim A$.}

\begin{lemma}\label{length and sup locally bounded} If $r>\frac{3}{2}$, then the following functions are bounded on every $G_c^{r}$-metric ball $B_{G_c^{r}}(c_0, \rho)$ in $\mathcal{I}^r$ by a constant depending only on $r$, $c_0$ and $\rho$
    \begin{align*}
        &(\mathcal{I}^r, \dist_{G_c^r}) \rightarrow (\R, \abs{\cdot}) \ ; \ c \mapsto l_c \ , \\
        &(\mathcal{I}^r, \dist_{G_c^r}) \rightarrow (\R, \abs{\cdot}) \ ; \ c \mapsto l_c^{-1} \ ,\\
        &(\mathcal{I}^r, \dist_{G_c^r}) \rightarrow (L^\infty(\cir, \R), \norm{\cdot}_{L^\infty}) \ ; \ c \mapsto \abs{c_\theta} \ ,\\
        &(\mathcal{I}^r, \dist_{G_c^r}) \rightarrow (L^\infty(\cir, \R)\ , \norm{\cdot}_{L^\infty}) \ ; \ c \mapsto \abs{c_\theta}^{- 1}.
    \end{align*}
\end{lemma}

 It should be noted that one can actually show that, for $r>\frac{3}{2}$, the function $c\mapsto l_c$ is bounded on $G_c^q$-metric balls in $\mathcal{I}^r$ for $1\leq q \leq r$. This can be achieved by carefully following the proof below and suitably adapting Lemma \ref{gronwall lemma}. 

\begin{proof}
    Calculating the derivative of $c \mapsto l_c$ in the direction of $h \in T_c\mathcal{I}^{r}$ we acquire
    \begin{equation}
    D_{c,h}(l_c) = \frac{d}{dt}\bigg\vert_{t=0} l_{\sigma(t)} = \frac{d}{dt}\bigg\vert_{t=0} \left(\int_0^1 \abs{\sigma_\theta(t) } \ d\theta\right) = \int_{\cir} \ip{D_s h, D_s c} \ ds,
    \end{equation}
    where $D_s c = \frac{c_\theta}{\abs{c_\theta}}$ is the unit tangent vector to $c$. From this we estimate
    \begin{align*}
        \abs{D_{c,h}(l_c)} &= \left|\int_{\cir} \ip{D_s h, D_s c} \ ds\right| \leq \int_{\cir} \abs{D_s h} \ ds \leq l_c^{\frac{1}{2}} \norm{D_s h}_{G_c^0} = l_c^{\frac{1}{2}} \norm{h}_{\dot{G}_c^1},
    \end{align*}
    where the second inequality follows from H\"older inequality.
    If $l_c>1$ we have
    \begin{align*}
        \abs{D_{c,h}(l_c)} &\leq l_c^{\frac{1}{2}} \norm{h}_{\dot{G}_c^1} \leq l_c \norm{h}_{\dot{G}_c^1} \leq l_c \norm{h}_{{G}_c^r}.
    \end{align*}
    where in the last inequality we have used \eqref{homogeneous versus full norm nesting}. On the other hand, if $l_c \leq 1$ we have
    \begin{align*}
        \abs{D_{c,h}(l_c)} &\leq l_c^{\frac{1}{2}} \norm{h}_{\dot{G}_c^1} \leq {\norm{h}_{\dot{G}_c^1} \leq \norm{h}_{G_c^r}},
    \end{align*}
    where again in the last inequality we have used \eqref{homogeneous versus full norm nesting}. In either scenario the first result follows from Lemma \ref{gronwall lemma}.
    
    For the second estimate we compute
    \begin{align*}
        D_{c,h}(l_c^{-1}) &= l_c^{-2} D_{c,h}(l_c).
    \end{align*}
    If $l_c>1$, using the above we have
    \begin{align*}
        \abs{D_{c,h}(l_c^{-1})} &= l_c^{-2} \abs{D_{c,h}(l_c)} \leq l_c^{-2} l_c \norm{h}_{{G}_c^r} = l_c^{-1} \norm{h}_{{G}_c^r}.
    \end{align*}
    On the other hand, if $l_c \leq 1$ we estimate
    \begin{align*}
        \abs{D_{c,h}(l_c^{-1})} &= l_c^{-2} \abs{D_{c,h}(l_c)} \leq l_c^{-2} l_c^{\frac{1}{2}} \norm{h}_{\dot{G}_c^1} = l_c^{-2} \norm{h \circ \psi_c^{-1}}_{\dot{H}^1} \\
        &\leq l_c^{-2} \norm{h \circ \psi_c^{-1}}_{\dot{H}^r} = l_c^{-2} l_c^{r-\frac{1}{2}} \norm{h}_{\dot{G}_c^r} = l_c^{-1} l_c^{r-\frac{3}{2}} \norm{h}_{\dot{G}_c^r} \\
        &\leq l_c^{-1} \norm{h}_{\dot{G}_c^r},
    \end{align*}
    where in the second and third equalities we have used \eqref{invariant versus non-invariant}, in the second inequality we have used \eqref{non-invariant nesting} and in the final inequality we have used the fact that $r > \frac{3}{2}$. The result once again follows from Lemma \ref{gronwall lemma}.

    To establish the last two results simultaneously, we prove the boundedness of the map $c \mapsto \log \abs{c_\theta}$. Calculating the derivative in the direction of $h \in T_c\mathcal{I}^r$ we have
    \begin{align*}
        D_{c,h}(\log \abs{c_\theta}) = \frac{1}{\abs{c_\theta}}D_{c,h}(\abs{c_\theta}) = \ip{D_s h , D_s c}
        .
    \end{align*}
    From this we estimate
    \begin{align*}
        \norm{D_{c,h}(\log \abs{c_\theta})}_{L^\infty} &= \norm{\ip{D_s h , D_s c}}_{L^\infty} \leq \norm{D_s h}_{L^\infty} \lesssim \norm{D_s h}_{G_c^{r-1}} \lesssim \norm{h}_{G_c^r},
    \end{align*}
    where the last two inequalities follow from Lemma \ref{eq:Sob_embed}, \eqref{invariant nesting} and the boundedness of $l_c$ and $l_c^{-1}$ on $G_c^r$-metric balls. The result then follows from Lemma \ref{gronwall lemma}.
\end{proof}

We now prove a weaker version of Lemma \ref{lemma 1}. We show that, for $r>\frac{3}{2}$ and $0\leq p \leq 1$, on $G_c^r$-metric balls in $\mathcal{I}^r$ we have the equivalence $\norm{\cdot}_{\dot{G}_c^p} \simeq \norm{\cdot}_{\dot{H}^p}$; whereas Lemma \ref{lemma 1} is concerned with the equivalence $\norm{\cdot}_{{G}_c^r} \simeq \norm{\cdot}_{{H}^r}$.

\begin{lemma}\label{mini lemma 1}
    Let $r>\frac{3}{2}$ and $0\leq p \leq 1$. Then, for any $G_c^r$-metric ball $B_{G_c^r}(c_0, \rho)$ in $\mathcal{I}^r$, there exists an $\beta(c_0,\rho,{r}, p) >0$ such that, for all $c \in B_{G_c^r}(c_0,\rho)$ and all $h \in T_c\mathcal{I}^r$, we have
    \begin{equation}\label{mini equivalence}
        \beta^{-1}\norm{h}_{\dot{H}^p} \leq \norm{h}_{\dot{G}_c^p} \leq \beta \norm{h}_{\dot{H}^p}.
    \end{equation}
\end{lemma}

\begin{proof}
    From Lemma \ref{relating homogeneous norms} and \eqref{composition} we estimate
    \begin{align*}
        \norm{h}_{\dot{G}_c^{p}} &= l_c^{\frac{1}{2}-p} \norm{h \circ \psi_c^{-1}}_{\dot{H}^p} \\
        &\leq l_c^{\frac{1}{2}-p} \norm{\frac{\abs{c_\theta}}{l_c}}_{L^\infty}^{\frac{1-p}{2}} \norm{\frac{l_c}{\abs{c_\theta}}}_{L^\infty}^{\frac{p}{2}} \norm{h}_{\dot{H}^p} \\
        &= \norm{\abs{c_\theta}}_{L^\infty}^{\frac{1-p}{2}} \norm{\abs{c_\theta}^{-1}}_{L^\infty}^{\frac{p}{2}} \norm{h}_{\dot{H}^p}
    \end{align*}
    and
    \begin{align*}
        \norm{h}_{\dot{G}_c^{p}} &= l_c^{\frac{1}{2}-p} \norm{h \circ \psi_c^{-1}}_{\dot{H}^p} \\
        &\geq l_c^{\frac{1}{2}-p} \norm{\frac{l_c}{\abs{c_\theta}}}_{L^\infty}^{\frac{-1+p}{2}} \norm{\frac{\abs{c_\theta}}{l_c}}_{L^\infty}^{-\frac{p}{2}} \norm{h}_{\dot{H}^p} \\
        &= \norm{\abs{c_\theta}^{-1}}_{L^\infty}^{\frac{-1+p}{2}} \norm{\abs{c_\theta}}_{L^\infty}^{-\frac{p}{2}} \norm{h}_{\dot{H}^p}.
    \end{align*}
    The result then follows from Lemma \ref{length and sup locally bounded}.
\end{proof}
The next two lemmas will play key technical roles in the proof of Lemma \ref{lemma 1}.
\begin{lemma}\label{baseline}
    For {$r>3/2$} the following functions are bounded on every $G_c^r$-metric ball $B_{G_c^r}(c_0, \rho)$ in $\mathcal{I}^r$ by a constant depending only on $r, c_0$ and $\rho$:
    \begin{align*}
        &(\mathcal{I}^r, \dist_{G_c^r}) \rightarrow (\dot{H}^{{\tilde{r}}}(\cir, \R^d), \norm{\cdot}_{\dot{H}^{{\tilde{r}}}}) \ ; \ c \mapsto D_s c \ , \\
        &(\mathcal{I}^r, \dist_{G_c^r}) \rightarrow (\dot{H}^{{\tilde{r}}}(\cir, \R^d), \norm{\cdot}_{\dot{H}^{{\tilde{r}}}}) \ ; \ c \mapsto \abs{c_\theta}^{\pm 1},
        \end{align*}
    {where $\tilde{r} = \min\{r-1,1\}$.}
\end{lemma}

\begin{proof}
    Note that, for $r>\frac{3}{2}$, we have both $\frac{1}{2}<\tilde{r} \leq 1$ and $\tilde{r}+1 \leq r$.
    Computing the derivatives of both functions in the direction $h\in T_c\mathcal{I}^r$ we have
    \begin{equation}\label{unit vector derivative}
        D_{c,h}(D_s c) = D_s h - \ip{D_s h , D_s c}D_s c
    \end{equation}
    and
    \begin{equation}\label{speed derivative}
        D_{c,h}(\abs{c_\theta}) = \ip{D_s h, D_s c} \abs{c_\theta}.
    \end{equation}
    Applying the triangle inequality to \eqref{unit vector derivative} we have
    \begin{align*}
        \norm{D_{c,h}(D_s c)}_{\dot{H}^{\tilde{r}}} &\leq \norm{D_s h}_{\dot{H}^{\tilde{r}}} + \norm{\ip{D_s h, D_s c}D_s c}_{\dot{H}^{\tilde{r}}}.
    \end{align*}
    For the first term above we apply Lemma \ref{mini lemma 1} and obtain
    \begin{align*}
        \norm{D_s h}_{\dot{H}^{\tilde{r}}} &\simeq \norm{D_s h}_{\dot{G}_c^{\tilde{r}}} = \norm{h}_{\dot{G}_c^r} \leq \norm{h}_{G_c^r}.
    \end{align*}
    While, for the second term, noting that $\norm{D_s c}_{L^\infty}=1$, we apply \eqref{unit vector product} and obtain
    \begin{equation}\label{lemma 4.6 expansion 1}
        \begin{split}
        \norm{\ip{D_s h, D_s c}D_s c}_{\dot{H}^{\tilde{r}}} &\lesssim \norm{\ip{D_s h, D_s c}}_{\dot{H}^{\tilde{r}}} + \norm{\ip{D_s h, D_s c}}_{L^\infty} \norm{D_s c}_{\dot{H}^{\tilde{r}}} \\
        &\lesssim \norm{\ip{D_s h, D_s c}}_{\dot{H}^{\tilde{r}}} + \norm{D_s h}_{L^\infty} \norm{D_s c}_{L^\infty} \norm{D_s c}_{\dot{H}^{\tilde{r}}} \\
        &\lesssim \norm{\ip{D_s h, D_s c}}_{\dot{H}^{\tilde{r}}} + \norm{D_s h}_{L^\infty} \norm{D_s c}_{\dot{H}^{\tilde{r}}}.
        \end{split}
    \end{equation}
    We approach the term $\norm{\ip{D_s h, D_s c}}_{\dot{H}^{\tilde{r}}}$ in an identical fashion, applying \eqref{unit vector product}
    \begin{equation}\label{lemma 4.6 expansion 1.2}
        \norm{\ip{D_s h, D_s c}}_{\dot{H}^{\tilde{r}}} \lesssim \norm{D_s h}_{\dot{H}^{\tilde{r}}} + \norm{D_s h}_{L^\infty}\norm{D_s c}_{\dot{H}^{\tilde{r}}}.
    \end{equation}
    By Lemmas \ref{sobolev embedding} and \ref{length and sup locally bounded} {and \eqref{invariant nesting}} we have $\norm{D_s h}_{L^\infty} \lesssim \norm{D_s h}_{G_c^{\tilde{r}}} \leq \norm{h}_{G_c^r}$ on $G_c^r$-metric balls. Combining all of this, \eqref{lemma 4.6 expansion 1} takes the form
    \begin{align*}
        \norm{\ip{D_s h, D_s c}D_s c}_{\dot{H}^{\tilde{r}}} &\lesssim \norm{h}_{G_c^r} + \norm{D_s c}_{\dot{H}^{\tilde{r}}} \norm{h}_{G_c^r} + \norm{D_s c}_{\dot{H}^{\tilde{r}}} \norm{h}_{G_c^r} \\
        &\lesssim \left(1 + \norm{D_s c}_{\dot{H}^{\tilde{r}}} \right)\norm{h}_{G_c^r}.
    \end{align*}
    Hence, by Lemma \ref{gronwall lemma} we have that $\norm{D_s c}_{\dot{H}^{\tilde{r}}}$ is bounded on $G_c^r$-metric balls $B_{G_c^r}(c_0,\rho)$ in $\mathcal{I}^r$ by a constant depending only on $r, c_0$ and $\rho$.

    We now turn our attention to the map
    \begin{align*}
        &(\mathcal{I}^r, \dist_{G_c^r}) \rightarrow (\dot{H}^{{\tilde{r}}}(\cir, \R^d), \norm{\cdot}_{\dot{H}^{{\tilde{r}}}}) \ ; \ c \mapsto \abs{c_\theta}.
    \end{align*}
    Using \eqref{speed derivative} and applying \eqref{product} with $f=\ip{D_s h, D_s c}, g= \abs{c_\theta}, a=b=\tilde{r}>\frac{1}{2}$ we have
    \begin{equation}\label{lemma 4.6 expansion 2}
        \begin{split}
            \norm{D_{c,h}(\abs{c_\theta})}_{\dot{H}^{\tilde{r}}} &= \norm{\ip{D_s h, D_s c}\abs{c_\theta}}_{\dot{H}^{\tilde{r}}} \\
            &\lesssim \abs{\widehat{\ip{D_s h, D_s c}}(0)} \norm{\abs{c_\theta}}_{\dot{H}^{\tilde{r}}} + \abs{\widehat{\abs{c_\theta}}(0)}\norm{\ip{D_s h, D_s c}}_{\dot{H}^{\tilde{r}}} \\
            &\quad + \norm{\ip{D_s h, D_s c}}_{\dot{H}^{\tilde{r}}} \norm{\abs{c_\theta}}_{\dot{H}^{\tilde{r}}}.
        \end{split}
    \end{equation}
    
    For the first term of \eqref{lemma 4.6 expansion 2} we have, using the above, that
    \begin{align*}
        \abs{\widehat{\ip{D_s h, D_s c}}(0)} &\leq \norm{\ip{D_s h, D_s c}}_{L^\infty} \lesssim \norm{D_s h}_{L^\infty} \norm{D_s c}_{L^\infty}\leq \norm{h}_{G_c^r}.
    \end{align*}
    For the second term we estimate $\abs{\widehat{\abs{c_\theta}}(0)}\leq \norm{\abs{c_\theta}}_{L^\infty}$ and recall, by Lemma \ref{length and sup locally bounded}, that this is bounded on $G_c^r$-metric balls $B_{G_c^r}(c_0,\rho)$ by a constant depending only on $r, c_0$ and $\rho$. Recycling \eqref{lemma 4.6 expansion 1.2} and remarking that, from the above argument, we now have that $\norm{D_s c}_{\dot{H}^{\tilde{r}}}$ is bounded on $G_c^r$-metric balls $B_{G_c^r}(c_0,\rho)$ in $\mathcal{I}^r$ by a constant depending only on $r, c_0$ and $\rho$, we acquire
    \begin{align*}
        \norm{\ip{D_s h, D_s c}}_{\dot{H}^{\tilde{r}}} &\lesssim \norm{D_s h}_{\dot{H}^{\tilde{r}}} + \norm{D_s h}_{L^\infty}\norm{D_s c}_{\dot{H}^{\tilde{r}}} \lesssim \norm{h}_{G_c^r}.
    \end{align*}
    Combining all of this, \eqref{lemma 4.6 expansion 2} becomes
    \begin{equation}\label{lemma 4.6 expansion 2.1}
        \begin{split}
            \norm{D_{c,h}(\abs{c_\theta})}_{\dot{H}^{\tilde{r}}} &\lesssim \norm{\abs{c_\theta}}_{\dot{H}^{\tilde{r}}}\norm{h}_{G_c^r} + \norm{h}_{G_c^r} + \norm{\abs{c_\theta}}_{\dot{H}^{\tilde{r}}}\norm{h}_{G_c^r} \\
            &\lesssim \left(1+\norm{\abs{c_\theta}}_{\dot{H}^{\tilde{r}}}\right)\norm{h}_{G_c^r}.
        \end{split}
    \end{equation}
    Hence, by Lemma \ref{gronwall lemma} we have that $\norm{\abs{c_\theta}}_{\dot{H}^{\tilde{r}}}$ is bounded on $G_c^r$-metric balls $B_{G_c^r}(c_0,\rho)$ in $\mathcal{I}^r$ by a constant depending only on $r, c_0$ and $\rho$. Similarly, $\norm{\abs{c_\theta}^{-1}}_{\dot{H}^{\tilde{r}}}$ is bounded on $G_c^r$-metric balls $B_{G_c^r}(c_0,\rho)$ in $\mathcal{I}^r$ by a constant depending only on $r, c_0$ and $\rho$.
\end{proof}

\begin{lemma}\label{workhorse}
If $r>2$ with decomposition $r=p + n$ for some $0 < p \leq 1$ and $n\geq 2$ an integer, then, for $1\leq k \leq n$, the following functions are bounded on every $G_c^r$-metric ball $B_{G_c^r}(c_0, \rho)$ in $\mathcal{I}^r$ by a constant depending only on $r, p, n, k, c_0$ and $\rho$.
    \begin{align*}
        &(\mathcal{I}^r, \dist_{G_c^r}) \rightarrow (\dot{H}^{p}(\cir, \R), \norm{\cdot}_{\dot{H}^{p}}) \ ; \ c \mapsto D_s^k c \ ,\\
        &(\mathcal{I}^r, \dist_{G_c^r}) \rightarrow (\dot{H}^{p}(\cir, \R), \norm{\cdot}_{\dot{H}^{p}}) \ ; \ c \mapsto D_s^{k-1} \abs{c_\theta} \ ,\\
        &(\mathcal{I}^r, \dist_{G_c^r}) \rightarrow (\dot{H}^{p+{k-1}}(\cir, \R), \norm{\cdot}_{\dot{H}^{p+{k-1}}}) \ ; \ c \mapsto \abs{c_\theta}^{\pm 1}.
    \end{align*}
\end{lemma}

\begin{proof}
    Notice that the $k=1$ case for each function follows immediately from Lemma \ref{baseline}. 
    We proceed by induction on $k$. For $k\geq2$ recall the following formula from \cite[Lemma~3.3]{bruveris2015completeness}.
    \begin{equation}
        \begin{split}
        D_{c,h}\left( D_s^k c \right) &= D_s^k h - \ip{D_s^k h, D_s c}D_s c - k\ip{D_s h, D_s c}D_s^k c \\
        & \qquad - \ip{D_s h, D_s^k c}D_s c + \text{lower order terms,}
        \end{split}
    \end{equation}
    {where the lower order terms include only products of terms with less than $k$ derivative.}
    Applying the triangle inequality and ignoring the contributions of the lower order terms (one can readily show that these terms are bounded by $\norm{h}_{G_c^r}$ on $G_c^r$-metric balls up to constants depending only on $r, p, k, c_0$ and $\rho$) we have
    \begin{equation}\label{expansion 1}
        \begin{split}
        \norm{D_{c,h}\left( D_s^k c \right)}_{\dot{H}^p} &\lesssim \norm{D_s^k h}_{\dot{H}^p} + \norm{\ip{D_s^k h, D_s c}D_s c}_{\dot{H}^p} + k\norm{\ip{D_s h, D_s c}D_s^k c}_{\dot{H}^p} \\
        & \quad + \norm{\ip{D_s h, D_s^k c}D_s c}_{\dot{H}^p} +{\norm{h}_{G_c^r}}.
        \end{split}
    \end{equation}
    For the first term we apply Lemma \ref{mini lemma 1}, \eqref{invariant nesting} and Lemma \ref{length and sup locally bounded}
    \begin{align*}
        \norm{D_s^k h}_{\dot{H}^p} &\simeq \norm{D_s^k h}_{\dot{G}_c^p} = \norm{h}_{\dot{G}_c^{p+k}} \leq l_c^{r-p-k}\norm{h}_{\dot{G}_c^r} \lesssim \norm{h}_{\dot{G}_c^r} \leq \norm{h}_{G_c^r}.
    \end{align*}
    For the second term in \eqref{expansion 1} we first apply \eqref{product} with $f=\ip{D_s^k h, D_s c}$, $g=D_s c$, $a=p$ and $b=1$
    \begin{equation}\label{expansion 1.2}
    \begin{split}
        \norm{\ip{D_s^k h, D_s c}D_s c}_{\dot{H}^p} &\lesssim \abs{\widehat{\ip{D_s^k h, D_s c}}(0)} \norm{D_s c}_{\dot{H}^{p}} \\
        &\quad + \abs{\widehat{D_s c}(0)}\norm{\ip{D_s^k h, D_s c}}_{\dot{H}^{p}}\\
        &\quad +\norm{\ip{D_s^k h, D_s c}}_{\dot{H}^{p}}\norm{D_s c}_{\dot{H}^1}.
    \end{split}
    \end{equation}
    For the first term of \eqref{expansion 1.2} we apply the Cauchy-Schwartz inequality for the $L^2$-inner product
    \begin{align*}
        \abs{\widehat{\ip{D_s^k h, D_s c}}(0)} &= \abs{\int_{\cir}\ip{D_s^k h, D_s c} \, d\theta} 
        \lesssim \norm{D_s^k h}_{L^2} \norm{D_s c}_{L^2}.
    \end{align*}
    By Lemma \ref{mini lemma 1}, \eqref{invariant nesting} and Lemma \ref{length and sup locally bounded} we have
    \begin{align*}
        \norm{D_s^k h}_{L^2} &\simeq \norm{D_s^k h}_{G_c^0} = \norm{h}_{\dot{G}_c^k} \leq l_c^{r-k} \norm{h}_{\dot{G}_c^{r}} \lesssim \norm{h}_{\dot{G}_c^r} \leq \norm{h}_{G_c^r}.
    \end{align*}
    Recalling now that $D_s c$ is a unit vector we have $\norm{D_s c}_{L^2}=1$ and $\abs{\widehat{D_s c}(0)} \leq \norm{D_s c}_{L^\infty} = 1$.
    By Lemma \ref{baseline}, $\norm{D_s c}_{\dot{H}^1}$ is bounded on $G_c^r$-metric balls in $\mathcal{I}^r$ by constants depending only on $r, p, c_0$ and $\rho${, and thus, by \eqref{invariant nesting}, also $\norm{D_s c}_{\dot{H}^p}$}. Combining all this gives us $\abs{\widehat{\ip{D_s^k h, D_s c}}(0)} \lesssim \norm{h}_{G_c^r}$ and \eqref{expansion 1.2} becomes
    \begin{equation}\label{expansion 1.2.2}
        \norm{\ip{D_s^k h, D_s c}D_s c}_{\dot{H}^p} \lesssim \norm{\ip{D_s^k h, D_s c}}_{\dot{H}^{p}} + \norm{h}_{G_c^r}.
    \end{equation}
    For the first term in \eqref{expansion 1.2.2} we apply \eqref{product} with $f=D_s^k h, g= D_s c, a=p$ and $b=1$
    \begin{align*}
        \norm{\ip{D_s^k h, D_s c}}_{\dot{H}^{p}} &\lesssim \abs{\widehat{D_s^k h}(0)} \norm{D_s c}_{\dot{H}^{p}} + \abs{\widehat{D_s c}(0)}\norm{D_s^k h}_{\dot{H}^{p}}+\norm{D_s^k h}_{\dot{H}^{p}}\norm{D_s c}_{\dot{H}^1}.
    \end{align*}
    Similar to before we estimate using H{\"o}lder's inequality
    \begin{align*}
        \abs{\widehat{D_s^k h}(0)} = \abs{\int_{\cir} D_s^k h \ d\theta} \leq \norm{D_s^k h}_{L^1} \lesssim \norm{D_s^k h}_{L^2}.
    \end{align*}
    Recalling from above that $\norm{D_s^k h}_{L^2} \lesssim \norm{h}_{G_c^r}$ and $\norm{D_s^k h}_{\dot{H}^p} \lesssim \norm{h}_{G_c^r}$ and the boundedness of $\norm{D_s c}_{L^2}$, $\abs{\widehat{D_s c}(0)}$, $\norm{D_s c}_{\dot{H}^p}$ and $\norm{D_s c}_{\dot{H}^1}$ on $G_c^r$-metric balls in $\mathcal{I}^r$, \eqref{expansion 1.2.2} then becomes
    \begin{align*}
        \norm{\ip{D_s^k h, D_s c}D_s c}_{\dot{H}^p} \lesssim \norm{h}_{G_c^r}.
    \end{align*}
    The remaining terms in \eqref{expansion 1} are bounded in an almost identical fashion as
    \begin{align*}
        \norm{\ip{D_s h, D_s c}D_s^k c} \lesssim \norm{D_s^k c}_{\dot{H}^{p}}\norm{h}_{G_c^r}
    \end{align*}
    and
    \begin{align*}
        \norm{\ip{D_s h, D_s^k c}D_s c}_{\dot{H}^{p}} \lesssim \norm{D_s^k c}_{\dot{H}^{p}}\norm{h}_{G_c^r}.
    \end{align*}
    Hence $\norm{D_{c,h}\left( D_s^k c \right)}_{\dot{H}^{p}} \lesssim (1+ \norm{D_s^k c}_{\dot{H}^{p}})\norm{h}_{G_c^r}$ on $B_{G_c^r}(c_0,\rho)$ and the first result follows from Lemma \ref{gronwall lemma}.

    The boundedness of the second function on $G_c^r$-metric balls can be argued exactly as above using the formula from \cite[Lemma~3.3]{bruveris2015completeness}
    \begin{align}\label{expansion 2}
        D_{c,h}\left( D_s^{k-1} \abs{c_\theta} \right) &= \ip{D_s^{k} h, D_s c}\abs{c_\theta} - (k-2)\ip{D_s h, D_s c}D_s^{k-1} \abs{c_\theta} \\
        \nonumber &\qquad + \ip{D_s h, D_s^{k} c}\abs{c_\theta} + \text{lower order terms.}
    \end{align}
    
    Finally, for bounding the third function we use the boundedness of the second one, together with the expansion
    \begin{equation}\label{expansion 3}
        \partial_\theta^{{k-1}} \abs{c_\theta} = \sum_{j=1}^{k-2} \sum_{\alpha \in A_j} c_{j,\alpha} \prod_{i=0}^{k-2} \big(\partial_\theta^i \abs{c_\theta}\big)^{\alpha_i} D_s^{{k-1}} \abs{c_\theta},
    \end{equation}
    where $c_{j,\alpha}$ are constants and $\alpha = (\alpha_0,...,\alpha_{k-2})$ are multi-indices with index sets
    \begin{align*}
        A_j = \left\{ \alpha \ \left| \ \sum_{i=0}^{k-2} i \alpha_i = k-1-j \ \text{and} \ \sum_{i=0}^{k-2} \alpha_i = j\right.\right\}.
    \end{align*}
    Applying the triangle inequality and \eqref{product} to \eqref{expansion 3} and using an induction argument, we acquire $\norm{\abs{c_\theta}}_{\dot{H}^{p+k-1}} \lesssim 1$ on $B_{G_c^r}(c_0,\rho)$ for all $2 \leq k \leq n$. To establish the result for $c \mapsto \abs{c_\theta}^{-1}$, we simply apply the chain rule to express $\partial_\theta^{k-1} \abs{c_\theta}^{-1}$ as a linear combination of powers of $\abs{c_\theta}^{-1}$ and derivatives up to order ${k-1}$ of $\abs{c_\theta}$.
\end{proof}

Armed with the above, we are now ready to prove the central estimate.

\begin{proof}[Proof of Lemma ~\ref{lemma 1}]
    Firstly, note that $\norm{\cdot}_{G_c^r} \simeq \norm{\cdot}_{G_c^0} + \norm{\cdot}_{\dot{G}_c^r}$. The equivalence $\norm{\cdot}_{G_c^0} \simeq \norm{\cdot}_{L^2}$ on $G_c^r$-metric balls in $\mathcal{I}^r$ follows directly from Lemma \ref{mini lemma 1}. Hence, to establish the estimate, we need to show the equivalence of the homogeneous norms $\norm{\cdot}_{\dot{G}_c^r} \simeq \norm{\cdot}_{\dot{H}^r}$ on $G_c^r$-metric balls in $\mathcal{I}^r$.

    We begin with the case $\frac{3}{2}<r\leq 2$. As $\frac{1}{2}<{r-1}\leq 1$ we have, by Lemma \ref{mini lemma 1}, that $\norm{\cdot}_{\dot{G}_c^{{r-1}}} \simeq\norm{\cdot}_{\dot{H}^{{r-1}}}$ on $G_c^r$-metric balls in $\mathcal{I}^r$. From this we have
    \begin{align*}
        \norm{h}_{\dot{G}_c^{r}} &= \norm{D_s h}_{\dot{G}_c^{{r-1}}} \simeq\norm{D_s h}_{\dot{H}^{{r-1}}} = \norm{\abs{c_\theta}^{-1} h_\theta}_{\dot{H}^{{r-1}}}.
    \end{align*}
    Applying \eqref{product} with $f=\abs{c_\theta}^{-1}$, $g=h_\theta$ and $a=b=r-1>\frac{1}{2}$ we acquire
    \begin{align*}
        \norm{h}_{\dot{G}_c^{r}} &\simeq \norm{\abs{c_\theta}^{-1} h_\theta}_{\dot{H}^{{r-1}}} \\
        &\lesssim \abs{\widehat{\abs{c_\theta}^{-1}}(0)} \norm{h_\theta}_{\dot{H}^{{r-1}}} + \abs{\widehat{h_\theta}(0)}\norm{\abs{c_\theta}^{-1}}_{\dot{H}^{{r-1}}} + \norm{\abs{c_\theta}^{-1}}_{\dot{H}^{{r-1}}}\norm{h_\theta}_{\dot{H}^{{r-1}}}.
    \end{align*}
    Note now that $\abs{\widehat{\abs{c_\theta}^{-1}}(0)} \leq \norm{\abs{c_\theta}^{-1}}_{L^\infty}$ and $\widehat{h_\theta}(0) = 0$, which gives us
    \begin{align*}
        \norm{h}_{\dot{G}_c^{r}} &\lesssim \norm{\abs{c_\theta}^{-1}}_{L^\infty} \norm{h_\theta}_{\dot{H}^{{r-1}}} + \norm{\abs{c_\theta}^{-1}}_{\dot{H}^{{r-1}}}\norm{h_\theta}_{\dot{H}^{{r-1}}} \\
        &=\left(\norm{\abs{c_\theta}^{-1}}_{L^\infty} + \norm{\abs{c_\theta}^{-1}}_{\dot{H}^{{r-1}}}\right)\norm{h}_{\dot{H}^{r}}.
    \end{align*}
    By Lemmas \ref{length and sup locally bounded} and \ref{baseline}, we have that $\norm{\abs{c_\theta}^{-1}}_{L^\infty}$ and $\norm{\abs{c_\theta}^{-1}}_{\dot{H}^{{r-1}}}$ are bounded on $G_c^r$-metric balls $B_{G_c^r}(c_0,\rho)$ in $\mathcal{I}^r$ by constants depending only on $r, c_0$ and $\rho$. Hence we have
    \begin{align}\label{eq:4.2aux1}
        \norm{h}_{\dot{G}_c^{r}} &\lesssim \norm{h}_{\dot{H}^r}
    \end{align}
    on $G_c^r$-metric balls $B_{G_c^r}(c_0,\rho)$ in $\mathcal{I}^r$.

    For the other direction, note that
    \begin{align*}
        \norm{h}_{\dot{H}^r} &= \norm{h_\theta}_{\dot{H}^{r-1}} = \norm{\abs{c_\theta}D_s h}_{\dot{H}^{r-1}}.
    \end{align*}
    Mirroring the above, we apply \eqref{product} with $f=\abs{c_\theta}$, $g=D_s h$, $a=b=r-1>\frac{1}{2}$ and acquire
    \begin{align*}
        \norm{h}_{\dot{H}^r} &= \norm{\abs{c_{\theta}}D_s h}_{\dot{H}^{{r-1}}} \\
        &\lesssim \abs{\widehat{\abs{c_\theta}}(0)} \norm{D_s h}_{\dot{H}^{{r-1}}} + \abs{\widehat{D_s h}(0)}\norm{\abs{c_\theta}}_{\dot{H}^{{r-1}}} + \norm{\abs{c_\theta}}_{\dot{H}^{{r-1}}}\norm{D_s h}_{\dot{H}^{{r-1}}}.
    \end{align*}
    For the first term we again bound $\abs{\widehat{\abs{c_\theta}}(0)} \leq \norm{\abs{c_
    \theta}}_{L^\infty}$. For the term $\abs{\widehat{D_s h}(0)}$ we estimate
    \begin{align*}
        \abs{\widehat{D_s h}(0)} &= \abs{\int_{\cir} \abs{c_\theta}^{-1} h_\theta \ d\theta} \leq \norm{\abs{c_\theta}^{-1}}_{L^2} \norm{h}_{\dot{H}^1} \leq\norm{\abs{c_\theta}^{-1}}_{L^\infty} \norm{h}_{\dot{H}^1},
    \end{align*}
    where, in the first inequality, we have used Cauchy-Schwartz for the $L^2$ inner product. This gives us
    \begin{align*}
        \norm{h}_{\dot{H}^r} &\lesssim \norm{\abs{c_\theta}}_{L^\infty}\norm{D_s h}_{\dot{H}^{{r-1}}} + \norm{\abs{c_\theta}^{-1}}_{L^\infty}\norm{\abs{c_\theta}}_{\dot{H}^{{r-1}}}\norm{h}_{\dot{H}^1} + \norm{\abs{c_\theta}}_{\dot{H}^{{r-1}}}\norm{D_s h}_{\dot{H}^{{r-1}}}.
    \end{align*}
    Using Lemma \ref{mini lemma 1} we have $\norm{D_s h}_{\dot{H}^{{r-1}}} \simeq \norm{D_s h}_{\dot{G}_c^{r-1}} = \norm{h}_{\dot{G}_c^r}$ and $\norm{h}_{\dot{H}^1} \simeq \norm{h}_{\dot{G}_c^1} \leq l_c^{r-1} \norm{h}_{\dot{G}_c^{r}}$ where, in the final inequality, we have used \eqref{invariant nesting}. Hence we have
    \begin{align*}
        \norm{h}_{\dot{H}^r} &\lesssim \left(\norm{\abs{c_\theta}}_{L^\infty} + l_c^{r-1}\norm{\abs{c_\theta}^{-1}}_{L^\infty}\norm{\abs{c_\theta}}_{\dot{H}^{{r-1}}} + \norm{\abs{c_\theta}}_{\dot{H}^{{r-1}}}\right)\norm{h}_{\dot{G}_c^{r}}.
    \end{align*}
    Finally, by Lemmas \ref{length and sup locally bounded} and \ref{baseline}, $\norm{\abs{c_\theta}}_{L^\infty}, l_c$ and $\norm{\abs{c_\theta}}_{\dot{H}^{{r-1}}}$ are bounded on $G_c^r$-metric balls $B_{G_c^r}(c_0,\rho)$ in $\mathcal{I}^r$ by constants depending only on $r, c_0$ and $\rho$. Hence we have
    \begin{align}\label{eq:4.2aux2}
        \norm{h}_{\dot{H}^r} \lesssim \norm{h}_{\dot{G}_c^r}
    \end{align}
    on $G_c^r$-metric balls $B_{G_c^r}(c_0,\rho)$ in $\mathcal{I}^r$. This delivers the lemma for the cases $\frac{3}{2}<r\leq 2$.
    
    Next consider $2<r$ with decomposition $r=p + n$ for some $0 < p \leq 1$ and $n\geq 2$ an integer. As $0<p\leq 1$ we have, by Lemma \ref{mini lemma 1}, that $\norm{\cdot}_{\dot{G}_c^{p}} \simeq\norm{\cdot}_{\dot{H}^{p}}$ on $G_c^r$-metric balls in $\mathcal{I}^r$. From this we have
    \begin{align*}
        \norm{h}_{\dot{G}_c^{p+1}} &= \norm{D_s h}_{\dot{G}_c^{p}} \simeq\norm{D_s h}_{\dot{H}^{p}} = \norm{\abs{c_\theta}^{-1} h_\theta}_{\dot{H}^{p}}.
    \end{align*}
    Repeating the same argument as for \eqref{eq:4.2aux1} (with a slight change of using $a=p$ and $b=1$ instead of $a=b=r-1$), we obtain
    \begin{align*}
        \norm{h}_{\dot{G}_c^{p+1}} &\lesssim \norm{h}_{\dot{H}^{p+1}}
    \end{align*}
    on $G_c^r$-metric balls $B_{G_c^r}(c_0,\rho)$ in $\mathcal{I}^r$.

    For the other direction, note that
    \begin{align*}
        \norm{h}_{\dot{H}^{p+1}} &= \norm{h_\theta}_{\dot{H}^p} = \norm{\abs{c_\theta}D_s h}_{\dot{H}^p}.
    \end{align*}
    We now repeat the same argument as for \eqref{eq:4.2aux2} (again with $a=p$ and $b=1$ instead of $a=b=r-1$), and obtain
    \begin{align*}
        \norm{h}_{\dot{H}^{p+1}} \lesssim \norm{h}_{\dot{G}_c^{p+1}}
    \end{align*}
    on $G_c^r$-metric balls $B_{G_c^r}(c_0,\rho)$ in $\mathcal{I}^r$.

    We now establish an inductive step. 
    Assume that, for some $k$ with $2 \leq k \leq n$ we have $\norm{h}_{\dot{G}_c^{p+k-1}} \simeq \norm{h}_{\dot{H}^{p+k-1}}$ on $G_c^r$-metric balls in $\mathcal{I}^r$. From this we obtain
    \begin{align*}
        \norm{h}_{\dot{G}_c^{p+k}} &= \norm{D_s h}_{\dot{G}_c^{p+k-1}} \simeq\norm{D_s h}_{\dot{H}^{p+k-1}} = \norm{\abs{c_\theta}^{-1} h_\theta}_{\dot{H}^{p+k-1}},
    \end{align*}
    and 
    \begin{align*}
        \norm{h}_{\dot{H}^{p+k}} &= \norm{h_\theta}_{\dot{H}^{p+k-1}} = \norm{\abs{c_\theta}D_s h}_{\dot{H}^{p+k-1}}.
    \end{align*}
    By the same argument as above, with $a=b=p+k-1>1/2$, and using Lemma~\ref{workhorse} instead of Lemma~\ref{baseline} to bound $\norm{\abs{c_\theta}}_{\dot{H}^{p+k-1}}$ and $\norm{\abs{c_\theta}^{-1}}_{\dot{H}^{p+k-1}}$,  we obtain
    \begin{align*}
        \norm{h}_{\dot{G}_c^{p+k}} &\simeq \norm{h}_{\dot{H}^{p+k}}
    \end{align*}
    on $G_c^r$-metric balls $B_{G_c^r}(c_0,\rho)$ in $\mathcal{I}^r$. The result for $2<r$ now follows by induction on $k$.
\end{proof}

Using Lemma \ref{lemma 1} we can now relate the induced geodesic distance on $(\mathcal{I}^r, G_c^r)$ to the standard norm distance on the ambient linear space $(H^r, \norm{\cdot}_{H^r})$.

\begin{lemma}\label{lemma 2}
    Let $r>\frac{3}{2}$. Then, for every $G_c^r$-metric ball $B_{G_c^r}(c_0, \rho)$ and every $c_1, c_2 \in B_{G_c^r}(c_0,\rho)$, we have
    \begin{equation}\label{lower bound}
        \norm{c_2-c_1}_{H^r} \leq \alpha(r,c_0,4\rho) \dist_{G_c^r}(c_1,c_2),
    \end{equation}
    {where $\alpha$ is as in Lemma~\ref{lemma 1}.}
\end{lemma}

\begin{proof}
    Let $c_1, c_2 \in B_{G_c^r}(c_0,\rho)$ and $\sigma : [0,1] \rightarrow \mathcal{I}^r$ a piecewise smooth curve connecting them with $G_c^r$-length $L_{G_c^r}(\sigma) < \dist_{G_c^r}(c_1,c_2) + \varepsilon$. Then, for $\varepsilon < \rho$, we have:
    \begin{equation*}
        L_{G_c^r}(\sigma) < \dist_{G_c^r}(c_1,c_2) + \varepsilon < 3\rho.
    \end{equation*}
    Hence, as
    \begin{equation*}
        \dist_{G_c^r}(\sigma(t), c_0) \leq \dist_{G_c^r}(\sigma(t), c_1) + \dist_{G_c^r}(c_1, c_0) \leq L_{G_c^r}(\sigma) + \rho < 3\rho + \varepsilon < 4\rho ,
    \end{equation*} 
    we have that $\sigma([0,1]) \subset B_{G_c^r}(c_0,4\rho)$. Finally, applying Lemma \ref{lemma 1}, we have that
    \begin{align*}
        \norm{c_1-c_2}_{H^r} \leq L_{H^r}(\sigma) = \int_0^1 \norm{\dot{\sigma}}_{H^r} \ dt \leq \alpha(r,c_0,4\rho) \int_0^1 \norm{\dot{\sigma}}_{G_c^r} \ dt = \alpha L_{G_c^r}(\sigma).
    \end{align*}
    Taking an infimum over all such $\sigma$ yields \eqref{lower bound}.
\end{proof}

We are now ready to present the proof of the main theorem.
\begin{proof}[Proof of Theorem~\ref{main theorem}]
    (1) Let $\{c_n\}\subset \mathcal{I}^r$ be a $G_c^r$-Cauchy sequence. Then there exists $\rho>0$ such that $\{c_n\}$ is contained in some $B_{G_c^r}(c_0,\rho)$. Hence, by Lemma \ref{lemma 2} there exists $\alpha >0$ such that 
    \begin{equation}
        \norm{c_N - c_M}_{H^r} \leq \alpha \dist_{G_c^r}(c_N,c_M),
    \end{equation}
    for all $N, M \in \mathbb{N}$. So $\{c_n\}$ is Cauchy in $(H^r, \norm{\cdot}_{H^r})$ and converges to some $c_\infty \in H^r$. From Lemma \ref{gronwall lemma} we have that $\{\abs{(c_n)_\theta}^{-1}\}$ is bounded away from $0$. As $r>\frac{3}{2}$, $H^r$ convergence implies $C^1$ convergence and hence $c_\infty \in \mathcal{I}^r$. Finally, as $G_c^r$ is a strong metric, {$\dist_{G_c^r}$ induces the same topology as the manifold topology \cite{lang1999}, which is in our case the Hilbert space topology of $H^r$. Thus,} $\norm{c_n - c_\infty}_{H^r} \rightarrow 0$ implies that $\dist_{G_c^r}(c_n, c_\infty) \rightarrow 0$.

 (2) Proving geodesic convexity follows exactly as in the integer-order results, using the direct methods in the calculus of variations and utilizing the estimates of Lemma~\ref{workhorse}. See \cite[Section 5]{bruveris2015completeness} or \cite[Section 5.5]{bauer2020sobolev}.
    
    (3) Next, for geodesic completeness, note that although the Hopf-Rinow theorem does not hold in infinite dimensions \cite{grossman1965, mcalpin1965, atkin1975} one still has that metric completeness implies geodesic completeness for strong metrics \cite{lang1999}. Hence we immediately have that, for $q>\frac{3}{2}$, $(\mathcal{I}^q, G^q)$ is geodesically complete. 

 To extend geodesic completeness to the case $r>q$ we will apply a no-loss-no-gain argument in spatial regularity, as originally developed by Ebin and Marsden to prove local well-posedness of the incompressible Euler equation~\cite{ebin1970}. 
 In the context of the situation of the present article, the same argument has already been used in~\cite{bauer2018fractional} to prove the local well-posedness for the $G^q$-metric on the whole scale of Sobolev immersions $\mathcal{I}^r(\cir,\mathbb R^d)$ with $r>\frac32$ and $r-q\geq 0$. 
 The exact same argument yields the desired global existence for all initial conditions if $q>\frac32$. 
 Since the metric $G^q$ is invariant by reparametrization, the same also holds for the corresponding geodesic spray. This allows one to apply the no loss-no gain result as formulated in~\cite{bruveris2017regularity}. 
 Using this  we obtain that solving the geodesic equation for initial conditions in $\mathcal{I}^r(\cir,\mathbb R^d)$ with $r\geq q$ the corresponding solution (geodesic) in $\mathcal{I}^r(\cir,\mathbb R^d)$ exists for the same maximal time interval as in $\mathcal{I}^q(\cir,\mathbb R^d)$. As $\left(\mathcal{I}^q(\cir,\mathbb R^d), G^q\right)$ is geodesically complete by the results of the previous section, this implies that all geodesics exist for all time in $\mathcal{I}^q(\cir,\mathbb R^d)$ and thus also in $\mathcal{I}^r(\cir,\mathbb R^d)$. This concludes the proof of geodesic completeness for $\mathcal{I}^r(\cir,\mathbb R^d)$ with $r>q$ and consequently also for $r=\infty$, i.e., for the Fr\'echet manifold $\Imm(\cir,\mathbb R^d)$. 

(4) Let $\eta_0 \in C_c^\infty(\R)$ be a standard mollifier, and let $\eta(t,x) = \frac{1}{t} \eta_0(x/t)$. 
Fix $r'\in (3/2,r)$ with $r'\ge q$. 
Let $c_0\in \mathcal{I}^{r'}(\cir,\mathbb R^d)\setminus \mathcal{I}^r(\cir,\mathbb R^d)$, and define $c:[0,\epsilon]\to \mathcal{I}^{r'}(\cir,\mathbb R^d)$ by 
\[
c(t) = 
\begin{cases}
    c_0 & t=0 \\
    c_0 * \eta(t,\cdot) & t>0,
\end{cases}
\]
{where $*$ denotes the convolution operator.}

By standard theory of mollifiers, $c(t) \to c(0)$ as $t\to 0$ in $H^{r'}(\cir,\R^d)$; since $r'>3/2$, it is true also in $C^1(\cir,\R^d)$.
Therefore, for $\epsilon>0$ small enough, it follows that indeed $c(t)$ is an immersion for all $t\in [0,\epsilon]$.
In particular, $c(t)\in \Imm(\cir,\R^d)$ for any $t\in(0,\epsilon]$.
Now, since $c$ is a smooth curve on the complete metric space $(\mathcal{I}^{r'}(\cir,\mathbb \R^d),\dist_{G^{r'}})$, it has a finite $G^{r'}$-length, and thus also a finite $G^q$-length (e.g., using \eqref{invariant nesting} and the fact that length is uniformly bounded along the curve).
Now, consider the path $c|_{t\in (0,1]}$ in $(\mathcal{I}^r(\cir,\mathbb \R^d),\dist_{G^q})$.
By what we proved, it is a finite length path that leaves the space (at $t=0$), since $c_0\notin \mathcal{I}^r(\cir,\mathbb \R^d)$. 
Thus the space is incomplete, as long it was a metric space to begin with (i.e., if $q>1/2$).
This completes   the metric incompleteness proof.
Now, if $q>3/2$, we can repeat the same argument for $r'=q$;
this shows that $\mathcal{I}^r(\cir,\mathbb \R^d)$ is dense in $\mathcal{I}^q(\cir,\mathbb \R^d)$ with respect to $\dist_{G^q}$.
Since $(\mathcal{I}^q(\cir,\mathbb \R^d),\dist_{G^q})$ is complete, we obtain it is the completion of $\mathcal{I}^r(\cir,\mathbb \R^d)$.

(5) It remains to prove the statement on geodesic incompleteness for $q<\frac32$ and $r>\frac32$. Therefore we will follow a similar argument as in~\cite{bauer2012sobolev}, where geodesic incompleteness for integer order metrics on the space of immersions has been studied. Namely, we consider the space $\mathcal C$
of all concentric circles as a subset of $\mathcal{I}^r(\cir,\mathbb R^d)$, i.e.,
\begin{equation}
\mathcal C:=\left\{(r\operatorname{cos}(\theta),r\operatorname{sin}(\theta)): r\in\mathbb R_{>0}\right\}\subset \mathcal{I}^r(\cir,\mathbb R^d).
\end{equation}
A straight forward calculation shows that the space $\mathcal C$ equipped with the restriction of the $G^q$-metric is in fact a totally geodesic subset of $\mathcal{I}^r(\cir,\mathbb R^d)$.
Consequently, if we can show that $\mathcal C$ is geodesically incomplete for $q<\frac32$ this also implies that $\mathcal{I}^r(\cir,\mathbb R^d)$ is geodesically incomplete. 
Furthermore, since $\mathcal C$ is finite dimensional, by the theorem of Hopf-Rinow this can be reduced to proving metric incompleteness. This allows us to conclude the proof by showing that one can scale down a circle to zero with finite $G^q$-length. 
To this end, let $c:[0,1)\to \Imm(\cir,\R^2)$, $c(t)=(1-t)(\operatorname{cos}(\theta),\operatorname{sin}(\theta))$. If $q<\frac{3}{2}$, a straightforward calculation yields the following inequality
\begin{align}
\int_0^1 G_c^q(\partial_t c, \partial_t c)^{1/2} dt \lesssim \frac{1}{\tfrac{3}{2}-q},
\end{align}
from which it follows that length of $c$ is finite.
Hence we have constructed a path of finite length that leaves the space $\mathcal C$. This yields the desired metric and geodesic incompleteness result of $\mathcal C$ and consequently geodesic incompleteness of $\mathcal{I}^r(\cir,\mathbb R^d)$. 
\end{proof}

\appendix
\section{Products and compositions in fractional Sobolev spaces}\label{appendix}
Here we provide details for the proof of the estimates given in Lemma \ref{lemma:com-pro}. As mentioned in Section \ref{sec:prelim}, our approach closely follows that of \cite{escherkolev2014}.
\begin{proof}[Proof of Lemma~\ref{lemma:com-pro}] \ \\
    \begin{enumerate}[(i)]
        \item The inequality \eqref{non-invariant nesting} is immediate from the definition \eqref{hom_sobolev norm}. \\
        \item The proof of the inequality \eqref{full norm product} can be found in \cite[Lemma~B.1]{escherkolev2014}. \\
        \item The proof of the \eqref{product} is essentially identical to the proof of \eqref{full norm product}. However, we record it here for completeness. We will deal explicitly with the case $d=1$; the extension to $d>1$ is straightforward.\\

        For $0< a \leq b$, note that from \eqref{hom_sobolev norm} we have
        \begin{align*}
            \norm{fg}_{\dot{H}^a}^2 &= \sum_{n \in \Z \setminus \{0\}} \abs{n}^{2a} \sabs{\widehat{fg}(n)}^2,
        \end{align*}
        where
        \begin{align*}
            \widehat{fg}(n) = \hat{f} * \hat{g} (n) = \sum_{j+k=n}\hat{f}(j)\hat{g}(k).
        \end{align*}
        This gives the inequality
        \begin{align*}
            \abs{n}^{a}\sabs{\widehat{fg}(n)} &\leq \sum_{j+k=n} \abs{j+k}^{a} \sabs{\hat{f}(j)}\abs{\hat{g}(k)} \\
            &\leq \sum_{\substack{j+k=n \\ \abs{j} \leq \abs{k}}} \abs{j+k}^{a} \sabs{\hat{f}(j)}\abs{\hat{g}(k)} + \sum_{\substack{j+k=n \\ \abs{j} > \abs{k}}} \abs{j+k}^{a} \sabs{\hat{f}(j)}\abs{\hat{g}(k)},
            \end{align*}
            which, up to a multiplication by $2^a$, gives us
            \begin{align*}
            \abs{n}^{a}\sabs{\widehat{fg}(n)} &\lesssim \sum_{\substack{j+k=n \\ \abs{j} \leq \abs{k}}} \abs{k}^{a} \sabs{\hat{f}(j)}\abs{\hat{g}(k)} + \sum_{\substack{j+k=n \\ \abs{k} < \abs{j}}} \abs{j}^{a} \sabs{\hat{f}(j)}\abs{\hat{g}(k)}.
        \end{align*}
        Separating the terms with $j, k =0$ we acquire
        \begin{equation}\label{lemma 2.2 expansion 1}
            \begin{split}
            \abs{n}^{a}\sabs{\widehat{fg}(n)} &\leq \sabs{\hat{f}(0)} \abs{n}^{a} \abs{\hat{g}(n)} + \abs{\hat{g}(0)} \abs{n}^{a} \sabs{\hat{f}(n)} \\
            &\qquad + \sum_{\substack{j+k=n \\ 0<\abs{j} \leq \abs{k}}} \abs{k}^{a} \sabs{\hat{f}(j)}\abs{\hat{g}(k)} + \sum_{\substack{j+k=n \\ 0 < \abs{k} < \abs{j}}} \abs{j}^{a} \sabs{\hat{f}(j)}\abs{\hat{g}(k)}.
            \end{split}
        \end{equation}
        Focusing on the third term above, we have
        \begin{align*}
            \sum_{\substack{j+k=n \\ 0<\abs{j} \leq \abs{k}}} \abs{k}^{a} \sabs{\hat{f}(j)}\abs{\hat{g}(k)} &\leq \sum_{\substack{j+k=n \\ 0<\abs{j} \leq \abs{k}}} \abs{\frac{k}{j}}^{b-a} \abs{k}^{a} \sabs{\hat{f}(j)}\abs{\hat{g}(k)} \\
            &\leq \sum_{\substack{j+k=n \\ 0<\abs{j} \leq \abs{k}}} \frac{1}{\abs{j}^b} \abs{j}^a \sabs{\hat{f}(j)}  \abs{k}^{b} \abs{\hat{g}(k)} \\
            &\leq \sum_{\substack{j+k=n \\ j \neq 0}} \frac{1}{\abs{j}^b} \abs{j}^a \sabs{\hat{f}(j)}  \abs{k}^{b} \abs{\hat{g}(k)} \\
            &= ( \lambda_b \tilde{f}_a * \tilde{g}_b ) (n),
        \end{align*}
        where, for $n \in \Z \setminus \{0\}$, we define $\lambda_b(n) = \frac{1}{\abs{n}^{b}}$, $\tilde{f}_a(n)=\abs{n}^a \sabs{\hat{f}(n)}$ and $\tilde{g}_b(n)=\abs{n}^b \abs{\hat{g}(n)}$ along with $\lambda_b(0) = \tilde{f}_a(0)=\tilde{g}_b(0)=0$. We note here that, for $a>0$, we have the equality $\norm{f}_{\dot{H}^a} = \snorm{\tilde{f}_a}_{\ell^2}$.

        Next, for the fourth term of \eqref{lemma 2.2 expansion 1}, we have
        \begin{align*}
            \sum_{\substack{j+k=n \\ 0 < \abs{k} < \abs{j}}} \abs{j}^{a} \sabs{\hat{f}(j)}\abs{\hat{g}(k)} &= \sum_{\substack{j+k=n \\ 0 < \abs{k} < \abs{j}}} \abs{\frac{k}{k}}^{b} \abs{j}^{a} \sabs{\hat{f}(j)}\abs{\hat{g}(k)} \\
            &= \sum_{\substack{j+k=n \\ 0 < \abs{k} < \abs{j}}} \abs{j}^{a} \sabs{\hat{f}(j)} \frac{1}{\abs{k}^b} \abs{k}^b \abs{\hat{g}(k)} \\
            &\leq \sum_{\substack{j+k=n \\ k\neq 0}} \abs{j}^{a} \sabs{\hat{f}(j)} \frac{1}{\abs{k}^b} \abs{k}^b \abs{\hat{g}(k)} \\
            &= \left(\tilde{f}_a * \lambda_b \tilde{g}_b \right) (n).
        \end{align*}
        Combining all this, \eqref{lemma 2.2 expansion 1} becomes
        \begin{equation}
            \abs{n}^{a}\sabs{\widehat{fg}(n)} \leq \sabs{\hat{f}(0)} \tilde{g}_a(n) + \abs{\hat{g}(0)} \tilde{f}_a(n) + \left(\lambda_b \tilde{f}_a * \tilde{g}_b \right) (n) + \left( \tilde{f}_a * \lambda_b \tilde{g}_b \right) (n),
        \end{equation}
        which gives
        \begin{equation}\label{lemma 2.2 expansion 2}
                \norm{fg}_{\dot{H}^a} \leq \sabs{\hat{f}(0)} \norm{\tilde{g}_a}_{\ell^2} + \abs{\hat{g}(0)} \snorm{\tilde{f}_a}_{\ell^2} + \snorm{\lambda_b \tilde{f}_a * \tilde{g}_b}_{\ell^2} + \snorm{\tilde{f}_a * \lambda_b \tilde{g}_b}_{\ell^2}.
        \end{equation}
        By Young's inequality and the Cauchy-Schwartz inequality for the $\ell^2$-inner product we have
        \begin{align*}
            \snorm{\lambda_b \tilde{f}_a * \tilde{g}_b}_{\ell^2} &\lesssim \snorm{\lambda_b \tilde{f}_a}_{\ell^1} \snorm{\tilde{g}_b}_{\ell^2} \lesssim \snorm{\lambda_b}_{\ell^2}\snorm{\tilde{f}_a}_{\ell^2} \snorm{\tilde{g}_b}_{\ell^2} \lesssim \snorm{f}_{\dot{H}^a}\snorm{g}_{\dot{H}^b}
        \end{align*}
        and
        \begin{align*}
            \snorm{\tilde{f}_a * \lambda_b \tilde{g}_b}_{\ell^2} &\lesssim \snorm{\tilde{f}_a}_{\ell^2} \snorm{\lambda_b \tilde{g}_b}_{\ell^1} \lesssim \snorm{\tilde{f}_a}_{\ell^2}\snorm{\lambda_b}_{\ell^2} \snorm{\tilde{g}_b}_{\ell^2} \lesssim \snorm{f}_{\dot{H}^a}\snorm{g}_{\dot{H}^b}.
        \end{align*}
        Hence, from \eqref{lemma 2.2 expansion 2} we have
        \begin{align}
            \norm{f g}_{\dot{H}^a} \lesssim_{(a,b)} |\hat{f}(0)| \norm{g}_{\dot{H}^a} + \abs{\hat{g}(0)}\norm{f}_{\dot{H}^a}+\norm{f}_{\dot{H}^a}\norm{g}_{\dot{H}^b}.
        \end{align} \\
        
        \item Recalling the Gagliardo semi-norm, cf. \cite{dinezza2012hitchhiker}, we have
        \begin{equation}\label{gagliardo}
            \norm{fg}_{\dot{H}^a} \simeq \int_{\cir}\int_{\cir} \frac{\abs{f(\theta)g(\theta)-f(\alpha)g(\alpha)}}{\abs{\theta-\alpha}^{1+2a}} \ d\theta d\alpha.
        \end{equation}
        Applying the triangle inequality, we have
        \begin{align*}
            \abs{f(\theta)g(\theta)-f(\alpha)g(\alpha)} &= \abs{f(\theta)g(\theta) - f(\theta)g(\alpha) + f(\theta)g(\alpha)-f(\alpha)g(\alpha)} \\
            &\leq \abs{f(\theta)g(\theta) - f(\theta)g(\alpha)} + \abs{f(\theta)g(\alpha)-f(\alpha)g(\alpha)} \\
            &\leq \norm{f}_{L^\infty} \abs{g(\theta)-g(\alpha)} + \norm{g}_{L^\infty} \abs{f(\theta)-f(\alpha)}.
        \end{align*}
        Substituting this into \eqref{gagliardo} immediately yields \eqref{unit vector product}. \\
        
        \item Applying a change of variables immediately gives
        \begin{align*}
            \norm{f \circ \phi}_{L^2} &\leq \norm{(\phi^{-1})_\theta}^{\frac{1}{2}} \norm{f}_{L^2}
        \end{align*}
        and
        \begin{align*}
            \norm{f \circ \phi}_{\dot{H}^1} & \leq \norm{\phi_\theta}^{\frac{1}{2}} \norm{f}_{\dot{H}^1}.
        \end{align*}
        The inequality \eqref{composition} then follows by interpolation, cf. \cite[Corollary 8.3]{frazier1990a} or \cite{triebel1983theory}.
    \end{enumerate}
\end{proof}

\bibliographystyle{abbrv}
\bibliography{bibliography}

\begin{thebibliography}{10}

\bibitem{arnold2021topological}
V.~Arnold and B.~Khesin.
\newblock {\em Topological {M}ethods in {H}ydrodynamics}.
\newblock Springer, Cham, {S}econd edition, 2021.

\bibitem{atkin1975}
C.~J. Atkin.
\newblock The {H}opf-{R}inow theorem is false in infinite dimensions.
\newblock {\em Bull. London Math. Soc.}, 7(3):261--266, 1975.

\bibitem{bauer2012vanishing}
M.~Bauer, M.~Bruveris, P.~Harms, and P.~Michor.
\newblock Vanishing geodesic distance for the {Riemannian} metric with geodesic
  equation the {KdV}-equation.
\newblock {\em Annals of Global Analysis and Geometry}, 41:461--472, 2012.

\bibitem{bauer2013geodesic}
M.~Bauer, M.~Bruveris, P.~Harms, and P.~Michor.
\newblock Geodesic distance for right invariant {Sobolev} metrics of fractional
  order on the diffeomorphism group.
\newblock {\em Annals of Global Analysis and Geometry}, 44(1):5--21, 2013.

\bibitem{bauer2018fractional}
M.~Bauer, M.~Bruveris, and B.~Kolev.
\newblock Fractional {S}obolev metrics on spaces of immersed curves.
\newblock {\em Calc. Var. Partial Differential Equations}, 57(1):Paper No. 27,
  24, 2018.

\bibitem{bauer2014overview}
M.~Bauer, M.~Bruveris, and P.~Michor.
\newblock Overview of the geometries of shape spaces and diffeomorphism groups.
\newblock {\em Journal of Mathematical Imaging and Vision}, 50:60--97, 2014.

\bibitem{bauer2012sobolev}
M.~Bauer, P.~Harms, and P.~Michor.
\newblock Sobolev metrics on shape space of surfaces.
\newblock {\em Journal of Geometric Mechanics}, 3(4):389--438, 2012.

\bibitem{bauer2023regularity}
M.~Bauer, P.~Harms, and P.~Michor.
\newblock Regularity and completeness of half-lie groups.
\newblock {\em Journal of the European Mathematical Society}, forthcoming.

\bibitem{bauer2016geometric}
M.~Bauer, B.~Kolev, and S.~Preston.
\newblock Geometric investigations of a vorticity model equation.
\newblock {\em Journal of Differential Equations}, 260(1):478--516, 2016.

\bibitem{bauer2020geodesic}
M.~Bauer, B.~Kolev, and S.~Preston.
\newblock Geodesic completeness of the {$H^{3/2}$} metric on
  {$\operatorname{Diff}(S^1)$}.
\newblock {\em Monatshefte f{\"u}r Mathematik}, 193(2):233--245, 2020.

\bibitem{bauer2020sobolev}
M.~Bauer, C.~Maor, and P.~Michor.
\newblock Sobolev metrics on spaces of manifold valued curves.
\newblock {\em Annali della Scuola Normale Superiore di Pisa}, Vol.
  XXIV:1895--1948, 2023.

\bibitem{bruveris2015completeness}
M.~Bruveris.
\newblock Completeness properties of {S}obolev metrics on the space of curves.
\newblock {\em J. Geom. Mech.}, 7(2):125--150, 2015.

\bibitem{bruveris2017regularity}
M.~Bruveris.
\newblock Regularity of maps between {Sobolev} spaces.
\newblock {\em Annals of Global Analysis and Geometry}, 52:11--24, 2017.

\bibitem{bruveris2014geodesic}
M.~Bruveris, P.~Michor, and D.~Mumford.
\newblock Geodesic completeness for {Sobolev} metrics on the space of immersed
  plane curves.
\newblock In {\em Forum of Mathematics, Sigma}, volume~2, page e19. Cambridge
  University Press, 2014.

\bibitem{bruveris2017completeness2}
M.~Bruveris and J.~M{\o}ller-Andersen.
\newblock Completeness of length-weighted {Sobolev} metrics on the space of
  curves.
\newblock {\em arXiv preprint arXiv:1705.07976}, pages 1--15, 2017.

\bibitem{bruveris2017completeness}
M.~Bruveris and F.~Vialard.
\newblock On completeness of groups of diffeomorphisms.
\newblock {\em Journal of the European Mathematical Society}, 19(5):1507--1544,
  2017.

\bibitem{camassa1993integrable}
R.~Camassa and D.~Holm.
\newblock An integrable shallow water equation with peaked solitons.
\newblock {\em Physical review letters}, 71(11):1661, 1993.

\bibitem{constantin1998wave}
A.~Constantin and J.~Escher.
\newblock Wave breaking for nonlinear nonlocal shallow water equations.
\newblock {\em Acta Mathematica}, 181(2):229--243, 1998.

\bibitem{constantin2002geometric}
A.~Constantin and B.~Kolev.
\newblock On the geometric approach to the motion of inertial mechanical
  systems.
\newblock {\em Journal of Physics A: Mathematical and General}, 35(32):R51,
  2002.

\bibitem{dinezza2012hitchhiker}
E.~Di~Nezza, G.~Palatucci, and E.~Valdinoci.
\newblock Hitchhiker's guide to the fractional {S}obolev spaces.
\newblock {\em Bull. Sci. Math.}, 136(5):521--573, 2012.

\bibitem{ebin1970}
D.~Ebin and J.~Marsden.
\newblock Groups of diffeomorphisms and the motion of an incompressible fluid.
\newblock {\em Ann. of Math. (2)}, 92:102--163, 1970.

\bibitem{eliashberg1993bi}
Y.~Eliashberg and L.~Polterovich.
\newblock Bi-invariant metrics on the group of {Hamiltonian} diffeomorphisms.
\newblock {\em Internat. J. Math}, 4(5):727--738, 1993.

\bibitem{escherkolev2014}
J.~Escher and B.~Kolev.
\newblock Right-invariant {S}obolev metrics of fractional order on the
  diffeomorphism group of the circle.
\newblock {\em J. Geom. Mech.}, 6(3):335--372, 2014.

\bibitem{escher2011geometry}
J.~Escher, B.~Kolev, and M.~Wunsch.
\newblock The geometry of a vorticity model equation.
\newblock {\em Communications on Pure and Applied Analysis}, 11(4):1407--1419,
  2011.

\bibitem{frazier1990a}
M.~Frazier and B.~Jawerth.
\newblock A discrete transform and decompositions of distribution spaces.
\newblock {\em J. Funct. Anal.}, 93(1):34--170, 1990.

\bibitem{grossman1965}
N.~Grossman.
\newblock Hilbert manifolds without epiconjugate points.
\newblock {\em Proc. Amer. Math. Soc.}, 16:1365--1371, 1965.

\bibitem{jerrard2019geodesic}
R.~Jerrard and C.~Maor.
\newblock Geodesic distance for right-invariant metrics on diffeomorphism
  groups: critical {Sobolev} exponents.
\newblock {\em Annals of Global Analysis and Geometry}, 56:351--360, 2019.

\bibitem{knappmann2023speed}
J.~Knappmann, H.~Schumacher, D.~Steenebr{\"u}gge, and H.~von~der Mosel.
\newblock A speed preserving {Hilbert} gradient flow for generalized integral
  {Menger} curvature.
\newblock {\em Advances in Calculus of Variations}, 16(3):597--635, 2023.

\bibitem{lang1999}
S.~Lang.
\newblock {\em Fundamentals of differential geometry}, volume 191 of {\em
  Graduate Texts in Mathematics}.
\newblock Springer-Verlag, New York, 1999.

\bibitem{leoni2017first}
G.~Leoni.
\newblock {\em A first course in {Sobolev} spaces}.
\newblock American Mathematical Soc., 2017.

\bibitem{leoni2023first}
G.~Leoni.
\newblock {\em A first course in fractional {Sobolev} spaces}, volume 229.
\newblock American Mathematical Society, 2023.

\bibitem{magnani2020remark}
V.~Magnani and D.~Tiberio.
\newblock A remark on vanishing geodesic distances in infinite dimensions.
\newblock {\em Proceedings of the American Mathematical Society},
  148(8):3653--3656, 2020.

\bibitem{marquis2018half}
T.~Marquis and K.-H. Neeb.
\newblock Half-lie groups.
\newblock {\em Transformation Groups}, 23(3):801--840, 2018.

\bibitem{mcalpin1965}
J.~McAlpin.
\newblock {\em Infinite dimensional manifolds and {M}orse theory}.
\newblock ProQuest LLC, Ann Arbor, MI, 1965.
\newblock Thesis (Ph.D.)--Columbia University.

\bibitem{michor2005vanishing}
P.~Michor and D.~Mumford.
\newblock Vanishing geodesic distance on spaces of submanifolds and
  diffeomorphisms.
\newblock {\em Documenta Mathematica}, 10:217--245, 2005.

\bibitem{michor2006riemannian}
P.~Michor and D.~Mumford.
\newblock Riemannian geometries on spaces of plane curves.
\newblock {\em Journal of the European Mathematical Society}, 8(1):1--48, 2006.

\bibitem{okabe2021convergence}
S.~Okabe and P.~Schrader.
\newblock Convergence of {Sobolev} gradient trajectories to elastica.
\newblock {\em arXiv preprint arXiv:2107.06504}, pages 1--29, 2021.

\bibitem{preston2018euler}
S.~Preston and P.~Washabaugh.
\newblock {Euler--Arnold} equations and {Teichm{\"u}ller} theory.
\newblock {\em Differential Geometry and its Applications}, 59:1--11, 2018.

\bibitem{reiter2021sobolev}
P.~Reiter and H.~Schumacher.
\newblock Sobolev gradients for the {M{\"o}bius} energy.
\newblock {\em Archive for Rational Mechanics and Analysis}, 242(2):701--746,
  2021.

\bibitem{schrader2023h}
P.~Schrader, G.~Wheeler, and V.~Wheeler.
\newblock On the {$H^1 (ds ^\gamma)$}-gradient flow for the length functional.
\newblock {\em The Journal of Geometric Analysis}, 33(9):297, 2023.

\bibitem{srivastava2016functional}
A.~Srivastava and E.~Klassen.
\newblock {\em Functional and shape data analysis}, volume~1.
\newblock Springer, 2016.

\bibitem{srivastava2010shape}
A.~Srivastava, E.~Klassen, S.~Joshi, and I.~Jermyn.
\newblock Shape analysis of elastic curves in euclidean spaces.
\newblock {\em IEEE Transactions on Pattern Analysis and Machine Intelligence},
  33(7):1415--1428, 2010.

\bibitem{triebel1983theory}
H.~Triebel.
\newblock {\em Theory of function spaces}.
\newblock Birkh{\"a}user, 1983.

\bibitem{washabaugh2016sqg}
P.~Washabaugh.
\newblock The {SQG} equation as a geodesic equation.
\newblock {\em Archive for Rational Mechanics and Analysis}, 222(3):1269--1284,
  2016.

\bibitem{wunsch2010geodesic}
M.~Wunsch.
\newblock On the geodesic flow on the group of diffeomorphisms of the circle
  with a fractional {Sobolev} right-invariant metric.
\newblock {\em Journal of Nonlinear Mathematical Physics}, 17(1):7--11, 2010.

\bibitem{younes2010shapes}
L.~Younes.
\newblock {\em Shapes and diffeomorphisms}, volume 171.
\newblock Springer, 2010.

\end{thebibliography}

\vfill

\end{document}